\documentclass[11pt]{article}

\usepackage{amsthm,amssymb,amsfonts,amsmath,graphicx, color,mathrsfs,booktabs}
\usepackage{hyperref}
\usepackage{xcolor}
\usepackage[margin=3cm]{geometry}

\usepackage{algorithm}
\usepackage{algpseudocode}

\newcommand{\ip}[2]{\langle #1 , #2 \rangle}    

\newcommand{\rad}{\mathfrak{r}}
\renewcommand{\v}{\mathfrak{v}}
\newcommand{\f}{\mathfrak{f}}
\renewcommand{\d}{\mathfrak{d}}

\def\aa{{\mathbf{a}}}
\def\bb{{\mathbf{b}}}

\def\dd{{\mathbf{d}}}
\def\ee{{\mathbf{e}}}

\def\gg{{\mathbf{g}}}

\def\ss{{\mathbf{s}}}

\def\uu{{\mathbf{u}}}
\def\vv{{\mathbf{v}}}
\def\ww{{\mathbf{w}}}
\def\xx{{\mathbf{x}}}
\def\yy{{\mathbf{y}}}
\def\zz{{\mathbf{z}}}

\def\aalpha{\boldsymbol{\alpha}}

\def\ssig{{\boldsymbol{\sigma}}}

\def\llambda{{\boldsymbol{\lambda}}}

\def\0{\mathbf{0}}
\def\1{\mathbf{1}}

\def\SS{{\mathbf{S}}}

\newcommand{\N}{\mathbb{N}}

\newcommand{\R}{\mathbb{R}}

\newcommand\cC{\mathcal{C}}

\newcommand\cO{{\ensuremath{\mathcal{O}}}}

\newcommand\cV{{\ensuremath{\mathcal{V}}}}

\DeclareMathOperator{\vertices}{vert}
\DeclareMathOperator{\relint}{ri}
\DeclareMathOperator{\relbnd}{rbd}

\DeclareMathOperator{\dist}{dist}
\DeclareMathOperator{\faces}{faces}

\DeclareMathOperator{\argmax}{arg\,max}
\DeclareMathOperator{\argmin}{arg\,min}
\DeclareMathOperator{\conv}{conv}
\DeclareMathOperator{\diam}{diam}

\DeclareMathOperator{\rank}{rank}

\theoremstyle{plain} \numberwithin{equation}{section}
\newtheorem{theorem}{Theorem}[section]
\numberwithin{theorem}{section}
\newtheorem{lemma}[theorem]{Lemma}
\newtheorem{corollary}[theorem]{Corollary}
\newtheorem{proposition}[theorem]{Proposition}
\newtheorem{definition}[theorem]{Definition}

\newcommand{\hrulealg}[0]{\vspace{1mm} \hrule \vspace{1mm}}

\usepackage{subcaption}
\begin{document}

\title{Fast convergence of Frank-Wolfe algorithms on polytopes }

\author{ Elias Wirth\thanks{Technical University of Berlin} \and Javier Pe\~na\thanks{Carnegie Mellon University} \and Sebastian Pokutta\thanks{Technical University of Berlin}}

\maketitle

\abstract{
We provide a template to derive  affine-invariant convergence rates for the following popular versions of the Frank-Wolfe algorithm on polytopes: vanilla Frank-Wolfe, Frank-Wolfe with away steps, Frank-Wolfe with blended pairwise steps, and Frank-Wolfe with in-face directions. Our template shows how the convergence rates follow from two affine-invariant properties of the problem, namely, {\em error bound} and {\em extended curvature}.  These properties depend solely on the polytope and objective function but not on any affine-dependent object like norms.  For each one of the above algorithms, we derive rates of convergence ranging from sublinear to linear depending on the degree of the error bound.
}

\section{Introduction}
We consider several popular versions of the {\em Frank-Wolfe algorithm}, also known as conditional gradient method, for optimization problems of the form
\begin{equation}\label{eq.opt}
    \min_{\xx\in\cC} f(\xx),
\end{equation}
where $\cC\subseteq \R^n$ is a polytope  and $f\colon \cC\to \R$ is a convex function differentiable in an open set containing $\cC$.  
At iteration $t$,  the {\em vanilla} Frank-Wolfe algorithm updates the current iterate $\xx^{(t)}\in \cC$ to
\[
\xx^{(t+1)} = \xx^{(t)} + \eta_t (\vv^{(t)}-\xx^{(t)})
\]
for $\vv^{(t)} = \argmin_{\yy\in \cC} \ip{\nabla f(\xx^{(t)})}{\yy}$ and some stepsize $\eta_t \in [0,1]$.  
This algorithmic scheme is attractive when projections onto $\cC$ are computationally expensive but a linear minimization oracle for the set $\cC$ is available.

When the pair $(\cC,f)$ satisfies a mild {\em curvature} property, the vanilla Frank-Wolfe algorithm has convergence rate $\cO(t^{-1})$, that is, 
\[f(\xx^{(t)}) - \min_{\xx\in\cC} f(\xx) = \cO(t^{-1}),\] 
see, for example,~\cite{jaggi2013revisiting}.  There are various settings that yield faster convergence rates~\cite{guelat1986some, levitin1966constrained, demianov1970approximate, dunn1979rates, garber2015faster, kerdreux2021projection, wirth2023acceleration, pena2023affine, wirth2023accelerated}.
However, in the setting often referred to as {\em Wolfe's lower bound}~\cite{wolfe1970convergence}, the convergence rate of the Frank-Wolfe algorithm cannot be faster than $\Omega(t^{-1-\epsilon})$ for any $\epsilon > 0$. The Wolfe's setting concerns the case when $\cC$ is a polytope, the optimal solution set 
lies in the relative interior of some proper face of $\cC$, and the stepsizes are chosen via exact line-search or short-step.  
 
Since optimization over polytopes is central to Frank-Wolfe research \cite{carderera2021cindy,tibshirani1996regression,joulin2014efficient}, 
this limitation has motivated  the development of linearly convergent variants for settings similar to Wolfe's \cite{wolfe1970convergence, holloway1974extension, lacoste2015global, garber2016linear, tsuji2022pairwise}. In contrast to the update in the vanilla Frank-Wolfe algorithm, these variants perform an update of the form
\[
\xx^{(t+1)} = \xx^{(t)} + \eta_t \dd^{(t)}
\]
for a more flexible choice of search direction $\dd^{(t)}$ that relies on additional information about the current iterate $\xx^{(t)}$.  This flexibility facilitates linear convergence rates at the cost of some computational overhead. 

Some popular variants are the Frank-Wolfe with away steps~\cite{beck2017linearly,lacoste2015global}, Frank-Wolfe with blended pairwise steps~\cite{tsuji2022pairwise}, and Frank-Wolfe with in-face directions~\cite{freund2017extended,garber2020revisiting,garber2016linear}.  These variants have stronger convergence properties than the vanilla  Frank-Wolfe algorithm.  In particular, under suitable smoothness and strong convexity assumptions, they have linear convergence, even in Wolfe's setting.  However, the linear convergence results derived in~\cite{beck2017linearly,lacoste2015global,tsuji2022pairwise,garber2016linear,garber2020revisiting} depend on the dimension of $\cC$, rely on additional difficult-to-verify assumptions, or are affine-dependent.  Furthermore, all of these papers consider only two regimes of convergence: a linear rate when suitable favorable assumptions hold, or the iconic rate $\cO(t^{-1})$ of first-order methods when only milder assumptions hold.

This paper develops a framework to derive convergence rates 
for the above variants of the Frank-Wolfe algorithm that circumvents the main limitations of previous approaches.  Our framework is founded on two key properties of the pair $(\cC,f)$, namely an {\em extended curvature property} and an {\em error bound property}.  We derive rates of convergence ranging from sublinear to linear depending on the degree of the error bound property.  Our approach yields the following main contributions.  First, we provide a common affine-invariant template to derive convergence rates for vanilla Frank-Wolfe, as well as the Frank-Wolfe variants that include away, blended pairwise, and in-face directions.  Second, we establish new convergence rates that interpolate between the $\cO(t^{-1})$ and linear rates considered in~\cite{beck2017linearly,lacoste2015global,tsuji2022pairwise,garber2016linear}.   The specific rate of convergence is determined by the degree of the error bound property in the same spirit developed in the recent articles~\cite{pena2023affine,wirth2023accelerated} for vanilla Frank-Wolfe.
Third, we sharpen the results in~\cite{lacoste2015global} by providing convergence results in terms of the local {\em facial distance}, in contrast to the dependence on the pyramidal width.  The former depends only on local geometry of the optimal face of $\cC$ whereas the latter depends on the global geometry of the entire polytope $\cC$.  Fourth, we establish dimension-independent rates for the aforementioned variants of Frank-Wolfe without requiring any additional difficult-to-verify assumptions such as strict complementarity~\cite{garber2020revisiting,wolfe1970convergence}.  Fifth, our developments enables us to subsume and extend the work of Garber and Meshi~\cite{garber2016linear} for simplex-like polytopes.

\subsection{Outline}\label{sec.outline}
The paper is structured as follows. In Section~\ref{sec.prelims}, we introduce the preliminaries and notation used throughout the paper.   In particular, we introduce the two key building blocks underlying our approach, namely, the extended curvature and error bound properties.
Section~\ref{sec.algorithms} presents our main developments, that is, the convergence results for vanilla Frank-Wolfe and Frank-Wolfe variants. The main convergence results are Theorem~\ref{thm.vanilla} for vanilla Frank-Wolfe, Theorem~\ref{thm.variants} for away-step and blended pairwise Frank-Wolfe, and Theorem~\ref{thm.in_face_variants} for in-face Frank-Wolfe.  
We deliberately use a similar format in each of the subsections of Section~\ref{sec.algorithms} to highlight how the convergence rates in all cases are a consequence of the {\em curvature} and {\em error bound} properties.

When the polytope $\cC$ has a standard form description of the form 
$\{\xx\in\R^n \mid A\xx = \bb, \xx\geq 0\}$, Theorem~\ref{thm.in_face_variants_polytope} gives a stronger version of Theorem~\ref{thm.in_face_variants}.   When the polytope $\cC$ is a simplex-like polytope, we provide a specialized version of the in-face Frank-Wolfe and establish a much stronger convergence result, namely Theorem~\ref{thm.in_face_variants.slp}.

In Section~\ref{sec.facial_geometry}, we develop some geometric properties of the facial structure of polytopes.  These properties are the crux of sufficient conditions for the error bound property to hold.

The proofs of  Theorem~\ref{thm.vanilla}, Theorem~\ref{thm.variants},  Theorem~\ref{thm.in_face_variants}, Theorem~\ref{thm.in_face_variants_polytope}, and Theorem~\ref{thm.in_face_variants.slp} are all instantiations of the following common template:
\[
\text{$L$-curvature and } (\mu,\theta)\text{-error bound} \Rightarrow \cO(t^{1/(2\theta-1)}) \text{ rate of convergence.}
\]
The expression $\cO(t^{1/(2\theta-1)})$ in the above template is to be interpreted as linear rate of convergence in the special case $\theta = 1/2$. 
The degree $\theta\in[0,1/2]$ of the error bound property yields a rate of convergence ranging from the sublinear rate $\cO(t^{-1})$ when $\theta = 0$ to a linear rate when $\theta = 1/2$.  The main constant hidden in the $\cO(t^{1/(2\theta-1)})$ expression is the ratio $L/\mu^{2\theta}$ where $L$ and $\mu$ are parameters associated to the curvature and error bound properties.  
As we detail in Section~\ref{sec.prelims}, the error bound property is defined in terms of a {\em distance function} $\d:\cC\times \cC\rightarrow [0,1]$.  The suitable distance function for each version of the Frank-Wolfe algorithm is as summarized in Table~\ref{table.thetable}.  The radial $\rad$, vertex $\v$, and face $\f$ distance functions are defined in Section~\ref{sec.algorithms}.

\begin{table}[!t]
\captionsetup[table]{justification=centering}
\centering
\caption{Frank-Wolfe version, suitable distance, and convergence result}
\begin{tabular}{c|c|c}
\hline
Frank-Wolfe version & Distance function $\d$ & Convergence result\\
\hline
vanilla & radial distance $\rad$ & Theorem~\ref{thm.vanilla} \\
away-step and blended pairwise & vertex distance $\v$& Theorem~\ref{thm.variants}\\
in-face & face distance $\f$& Theorem~\ref{thm.in_face_variants}\\
in-face and $\cC$ in standard form & face distance $\f$& Theorem~\ref{thm.in_face_variants_polytope}\\
in-face and $\cC$ simplex-like & face distance $\f$& Theorem~\ref{thm.in_face_variants.slp}\\
\hline
\end{tabular}\label{table.thetable}
\end{table}

\section{Preliminaries}\label{sec.prelims}
Throughout $\cC\subseteq \R^n$ denotes a nonempty polytope equipped with a {\em linear oracle} that computes the mapping
\[
\gg \mapsto \argmin_{\yy\in\cC} \ip{\gg}{\yy}.
\]
Furthermore, we shall assume that the minimizer returned by the above linear oracle is a vertex of $\cC$.

We shall also assume that $f\colon \cC\to \R$ is a convex function differentiable in an open set containing $\cC$.
For the optimization problem \eqref{eq.opt}, we denote the \emph{optimal value} by
\[
    f^* = \min_{\xx\in\cC} f(\xx)
\]
and the \emph{set of optimal solutions}
by
\[
    X^* : = \{\xx\in\cC \mid f(\xx) = f^*\}.
\]

Recall the following popular smoothness and H\"olderian error bound properties. These concepts are defined in terms of a norm on $\R^n$ and hence are typically  affine-dependent.
\begin{definition}[Smoothness]\label{def.smoothness} Let $\R^n$ be endowed with a norm $\|\cdot\|$, let $\cC \subseteq \R^n$ be a polytope, and let $f\colon \cC \to \R$ be convex and differentiable in an open set containing $\cC$. For $L>0$ we say that $f$ is \emph{$L$-smooth} over $\cC$ with respect to $\|\cdot\|$ if for all $\xx, \yy\in \cC$, it holds that
\[
        f(\yy) \leq f(\xx) + \langle \nabla f(\xx), \yy - \xx\rangle + \frac{L}{2}\|\xx-\yy\|^2.
\]
\end{definition}

For $\|\cdot\|,\cC,f$ as in Definition~\ref{def.smoothness} it is known and easy to show that $f$ is $L$-smooth over $\cC$ with respect to $\|\cdot\|$ if $\nabla f$ is $L$-Lipschitz continuous on $\cC$, that is, if
\[
\|\nabla f(\yy)-\nabla f(\xx)\|^* \le L \|\yy-\xx\| \text{ for all } \xx,\yy\in \cC
\]
where $\|\cdot\|^*$ denotes the dual norm of $\|\cdot\|$, that is, the norm on $\R^n$ defined as
\begin{equation}\label{eq.dual.norm}
\|\uu\|^*:=\max_{\|\xx\|\le 1} \ip{\uu}{\xx}.
\end{equation}

\begin{definition}[Hölderian error bound]\label{def.heb}
Let $\R^n$ be endowed with a norm $\|\cdot\|$, let $\cC \subseteq \R^n$ be a polytope, and let $f\colon \cC \to \R$ be convex and differentiable in an open set containing $\cC$.  For $\mu> 0$ and $\theta \in (0, 1/2]$ we say that $(\cC,f)$ satisfies the \emph{$(\mu, \theta)$-Hölderian error bound} over $\cC$ with respect to $\|\cdot\|$ if for all $\xx\in \cC$ it holds that
\[
        \left(\frac{f(\xx)-f^*}{\mu}\right)^\theta \geq \min_{\xx^*\in X^*} \|\xx^*-\xx\|.
\]
\end{definition}

For $\|\cdot\|,\cC,f$ as in Definition~\ref{def.heb}, and $\mu > 0$ and $\theta \in (0,1/2]$, it is easy to see that $(\cC,f)$ satisfies the $(\mu,\theta)$-Hölderian error bound over $\cC$ with respect to $\|\cdot\|$ if $f$ satisfies the following \emph{uniform convexity} property on $\cC$:
$$
f(\lambda \xx + (1-\lambda)\yy) 
\le \lambda f(\xx) + (1-\lambda)f(\yy) - \mu\lambda(1-\lambda)\|\xx-\yy\|^{1/\theta} 
\text{ for all } \xx,\yy\in\cC \text{ and } \lambda \in [0,1].
$$

In order to derive affine-invariant accelerated convergence rates for FW variants, we rely on the following generalizations of smoothness and H\"olderian error bound.  As we detail in Section~\ref{sec.algorithms}, the affine-invariant analysis of the methods presented in this paper relies on these concepts.

The following concepts of curvature and extended curvature were introduced by
Jaggi~\cite{jaggi2013revisiting} and by Lacoste-Julien and Jaggi~\cite{lacoste2015global} respectively.
\begin{definition}[Curvature]\label{def.error_bound_curvature}
    Let $\cC \subseteq \R^n$ be a polytope  and $f\colon \cC \to \R$ be convex and differentiable in an open set containing $\cC$.    \begin{enumerate}
        \item[(a)]
              For $L > 0$ we say that $(\cC,f)$ has \emph{$L$-curvature} if for all $\xx\in \cC, \; \dd \in \cC-\xx$, and $\eta \in [0,1]$ it holds that
              \begin{equation}\label{eq.step_size_motivator}
                  f(\xx + \eta \dd) \le  f(\xx) + \eta \ip{\nabla f(\xx)}{\dd} + \frac{L\eta^2}{2}.
              \end{equation}
        \item[(b)]
                      For $L > 0$ we say that $(\cC,f)$ has \emph{extended $L$-curvature} if for all $\xx\in \cC, \; \dd \in \cC-\cC$, and $\eta \ge 0$ such that $\xx + \eta \dd\in \cC$ it holds that
\[
                  f(\xx + \eta \dd) \le  f(\xx) + \eta \ip{\nabla f(\xx)}{\dd} + \frac{L\eta^2}{2}.
\]
    \end{enumerate}
\end{definition}

It is easy to see that the above $L$-curvature and extended $L$-curvature properties are affine invariant.  It is also evident that $(\cC,f)$ has
$L$-curvature whenever $(\cC,f)$ has extended
$L$-curvature. In addition, as noted by Jaggi~\cite{jaggi2013revisiting} and by Lacoste-Julien and Jaggi~\cite{lacoste2015global}, if $f$ is $L$-smooth over $\cC$ with respect to a norm $\|\cdot\|$ then $(\cC,f)$ has extended $L\cdot \diam(\cC)^2$-curvature, where $\diam(\cC)$ is the following {\em diameter} of the set $\cC$
\[
\diam(\cC) = \max_{\xx,\yy\in \cC} \|\xx-\yy\|.
\]

In addition to the above curvature concepts, we will rely on the following extended error bound concept.

\begin{definition}[Error bound]\label{def.error.bound}
    Let $\cC \subseteq \R^n$ be a polytope, $f\colon \cC \to \R$ be convex and differentiable in an open set containing $\cC$, and $\d:\cC\times \cC \rightarrow [0,1]$.
              For $\mu > 0$ and $\theta\in[0,1/2]$, we say that $(\cC,f)$ satisfies the \emph{$(\mu,\theta)$-error bound relative to $\d$} if for all $\xx\in \cC$  it holds that
\[
                  \left(\frac{f(\xx) - f^*}{\mu}\right)^{\theta} \ge \min_{\xx^*\in X^*} \d(\xx^*, \xx).
\]
For ease of notation, we will write $\d(X^*,\xx)$ for $\displaystyle\min_{\xx^*\in X^*} 
\d(\xx^*, \xx).$
\end{definition}

For $\cC, f, \d$ as in Definition~\ref{def.error.bound}, the pair
$(\cC,f)$ automatically satisfies the $(\mu,0)$-error bound relative to $\d$ for any $\mu > 0$ because $\d:\cC\times \cC\rightarrow [0,1]$.  The more interesting convergence results hold when $(\cC,f)$ satisfies the $(\mu,\theta)$-error bound relative $\d$ for some $\mu>0$, $\theta \in (0,1/2]$ and some judiciously chosen $\d$.

It is easy to see that the above error bound property is affine invariant provided the function $\d:\cC\times \cC \rightarrow [0,1]$ is affine-invariant.  Section~\ref{sec.algorithms} will feature the following instances of affine-invariant distance functions $\d:\cC\times \cC \rightarrow [0,1]$ for each of the algorithms we analyze: the {\em radial distance} for vanilla Frank-Wolfe, the {\em vertex distance}  for away-step Frank-Wolfe and blended pairwise Frank-Wolfe, and the {\em face distance}  for the in-face Frank-Wolfe.

We will also provide sufficient conditions for the error bound property to hold.  To that end, we will rely on the following geometric features of a polytope that were introduced in~\cite{pena2019polytope}.

Let $\cC\subseteq \R^n$ be a polytope.  Denote the sets of vertices and nonempty faces of $\cC$ by $\vertices(\cC)$ and $\faces(\cC)$ respectively.  
Suppose $\|\cdot\|$ is a norm on $\R^n$. 
For closed and nonempty $F\subseteq \cC, \, G \subseteq \cC$ let 
\[
    \dist(F, G) = \min_{\xx\in F, \yy\in G} \|\xx-\yy\|.
\]

\begin{definition}[Inner and outer facial distances]\label{def.facial_distance}
Let $\R^n$ be endowed with a norm $\|\cdot\|$ and let $\cC \subseteq \R^n$ be a polytope. For $F\in \faces(\cC)$ the \emph{inner facial distance} from $F$ is defined as
\[
        \Phi(F,\cC):=\min_{G\in\faces(F)\atop G\subsetneq \cC} \dist(G, \conv(\vertices(\cC)\setminus G)),
\]
and the  \emph{outer facial distance} from $F$  is defined as
\[
        \bar\Phi(F,\cC):=\min_{G\in\faces(F),H\in\faces(\cC)\atop G\cap H = \emptyset} \dist(G,H).
\]
\end{definition}
It is easy to see that $\bar \Phi(F,\cC) \ge \Phi(F,\cC) \ge \Phi(\cC,\cC) > 0$ for all $F\in \faces(\cC)$.  Furthermore, in~\cite{pena2019polytope} it was shown that $\Phi(\cC,\cC)$ coincides with the \emph{pyramidal width} of $\cC$ introduced by Lacoste-Julien and Jaggi~\cite{lacoste2015global}.

\medskip

Figure~\ref{fig:innerfacial} and Figure~\ref{fig:outerfacial} respectively illustrate the inner facial distance $\Phi(F,\cC)$ and the outer facial distance $\bar\Phi(F,\cC)$ for $F=\{(0,0)\}, \cC = [0,1]\times [0,1]$ and for 
$F=\{(0,0)\}, \cC = [0,2]\times [0,1]$ when $\R^2$ is endowed with the Euclidean norm.   In both cases  $F:=\{(0,0)\}\in \faces(\cC)$ and the relevant distance is depicted via the dotted segment(s).  In Figure~\ref{fig:innerfacial}(a) the dotted segment has length $\Phi(F,\cC)=\sqrt{2}/2$, in Figure~\ref{fig:innerfacial}(b) the dotted segment has length $\Phi(F,\cC)=2\sqrt{5}/5$.
In Figure~\ref{fig:outerfacial} all of the dotted segments have length $\bar \Phi(F,\cC)=1$. Observe that the outer distance $\bar \Phi(F,\cC)$ is attained at two different segments for $F=\{(0,0)\}, \cC = [0,1]\times [0,1]$.

\begin{figure}[!t]
\captionsetup[subfigure]{justification=centering}
 \hspace{-.6in}
\begin{tabular}{c c} 
\begin{subfigure}{.55\textwidth}
    \centering
        \includegraphics[width=.55\textwidth]{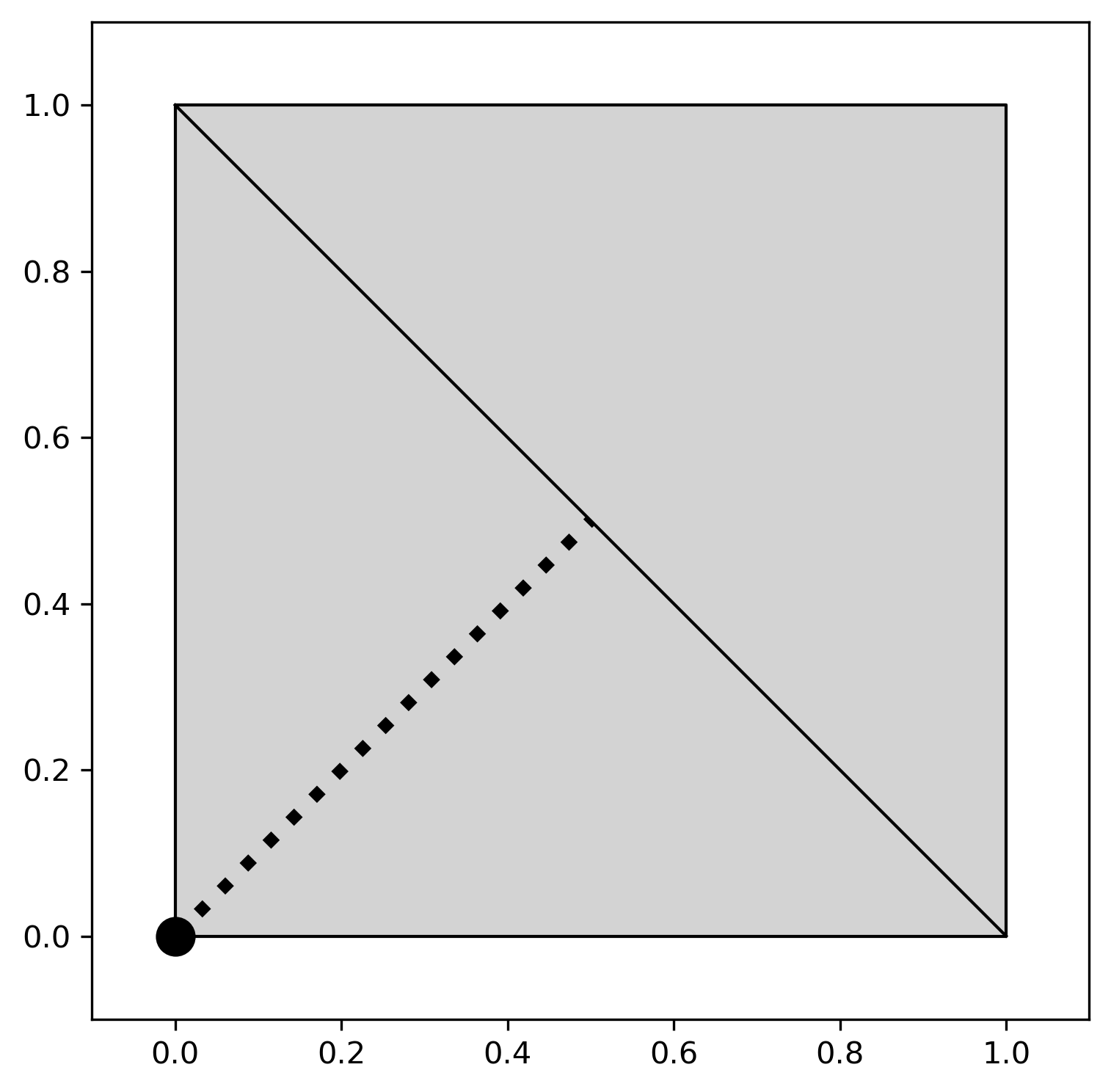}
        \caption{$F = \{(0,0)\},\;\cC = [0,1]\times [0,1], \; \Phi(F,\cC)=\sqrt{2}/2.$}
    \end{subfigure} & \hspace{-.5in}
    \begin{subfigure}{.55\textwidth}
    \centering
        \includegraphics[width=1\textwidth]{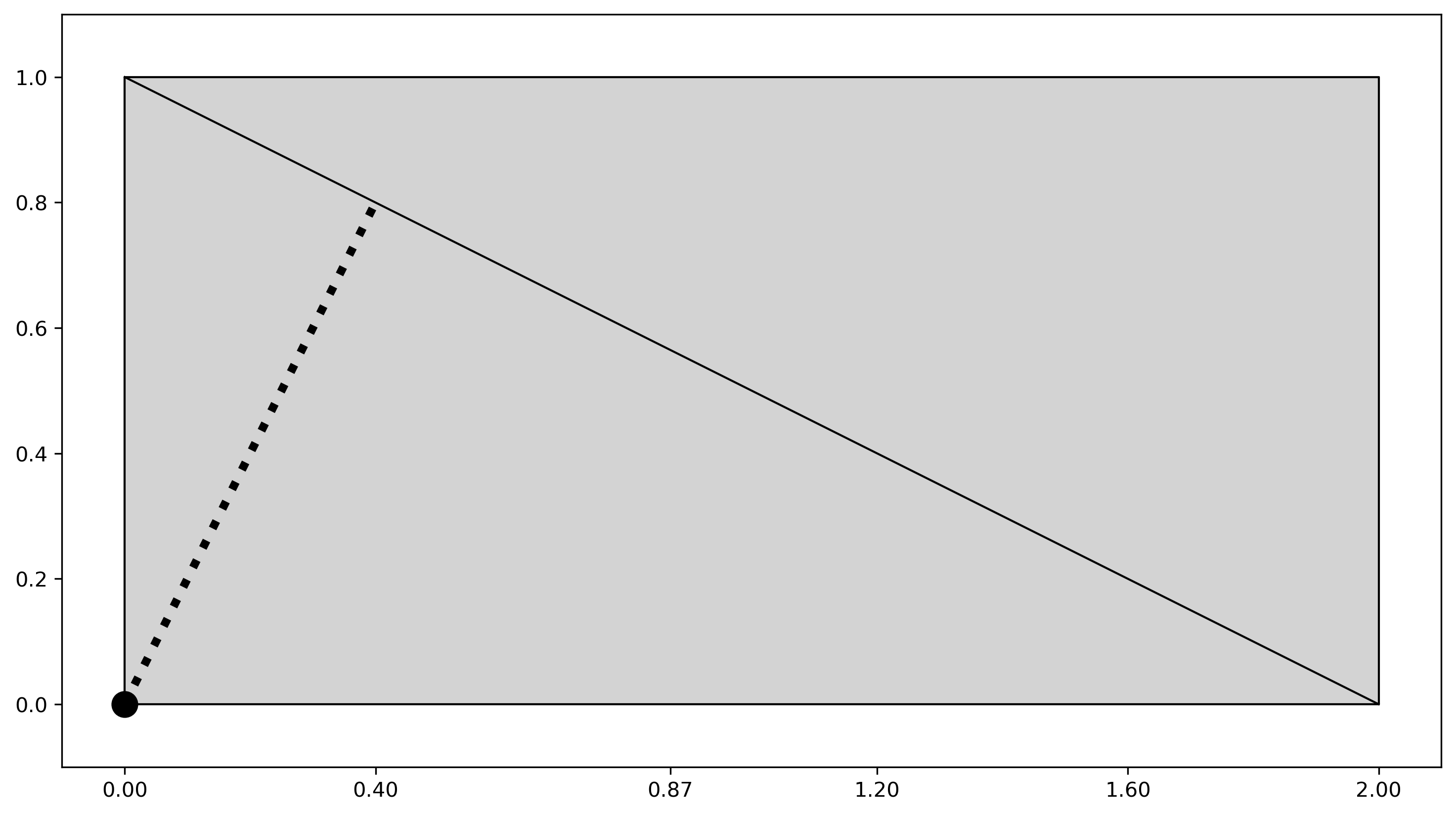}
        \caption{$F = \{(0,0)\},\;\cC = [0,2]\times [0,1], \; \Phi(F,\cC)=2\sqrt{5}/5$.}
    \end{subfigure} 
\end{tabular}
\caption{Inner facial distance $\Phi(F,\cC)$ equals the length of the dotted line.}\label{fig:innerfacial}
\end{figure}

\begin{figure}[!t]
\captionsetup[subfigure]{justification=centering}
 \hspace{-.6in}
\begin{tabular}{c c} 
\begin{subfigure}{.55\textwidth}
    \centering
        \includegraphics[width=.55\textwidth]{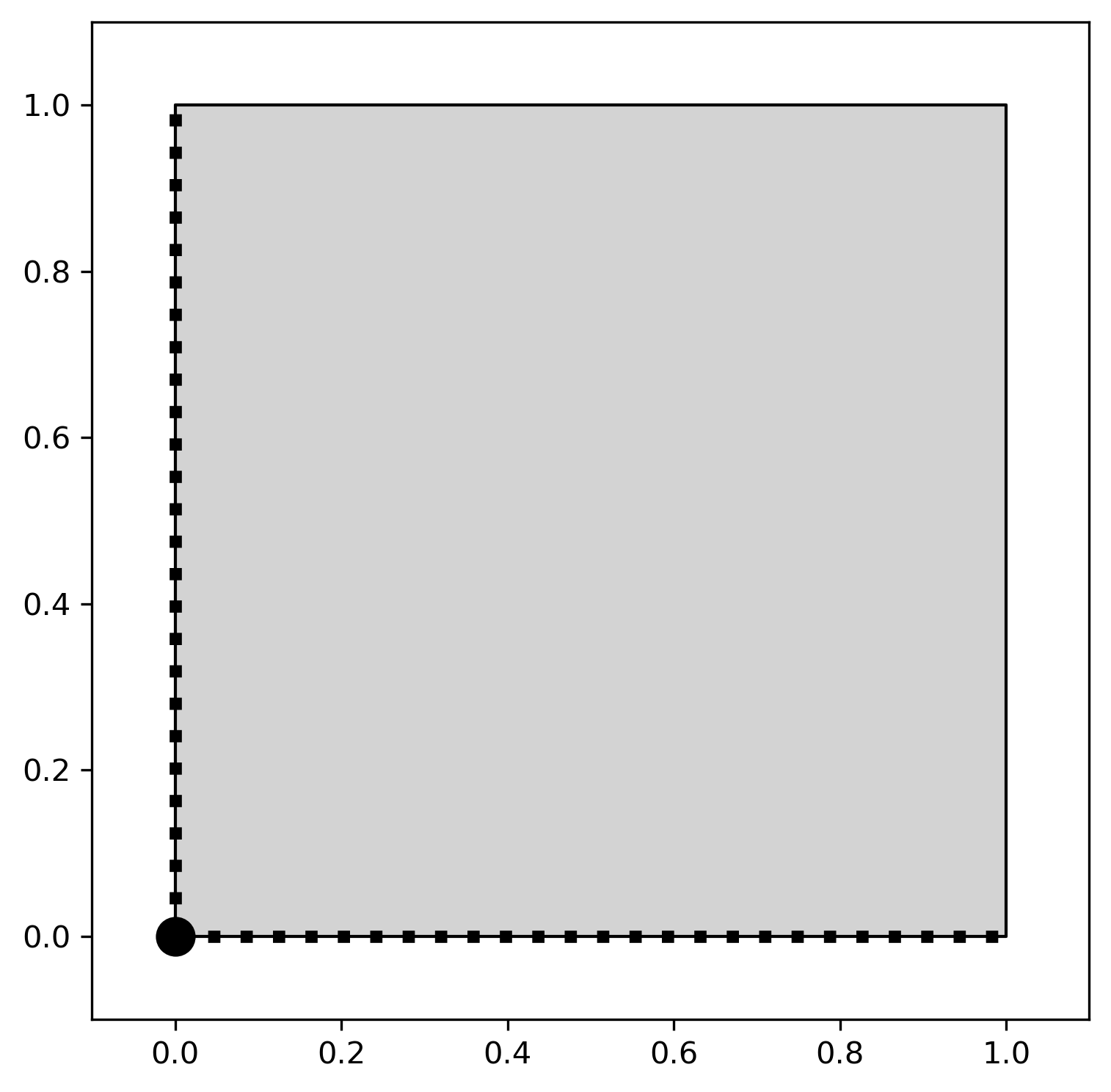}
        \caption{$F = \{(0,0)\},\;\cC = [0,1]\times [0,1], \; \bar \Phi(F,\cC)=1$.}
    \end{subfigure} & \hspace{-.5in}
    \begin{subfigure}{.55\textwidth}
    \centering
        \includegraphics[width=1\textwidth]{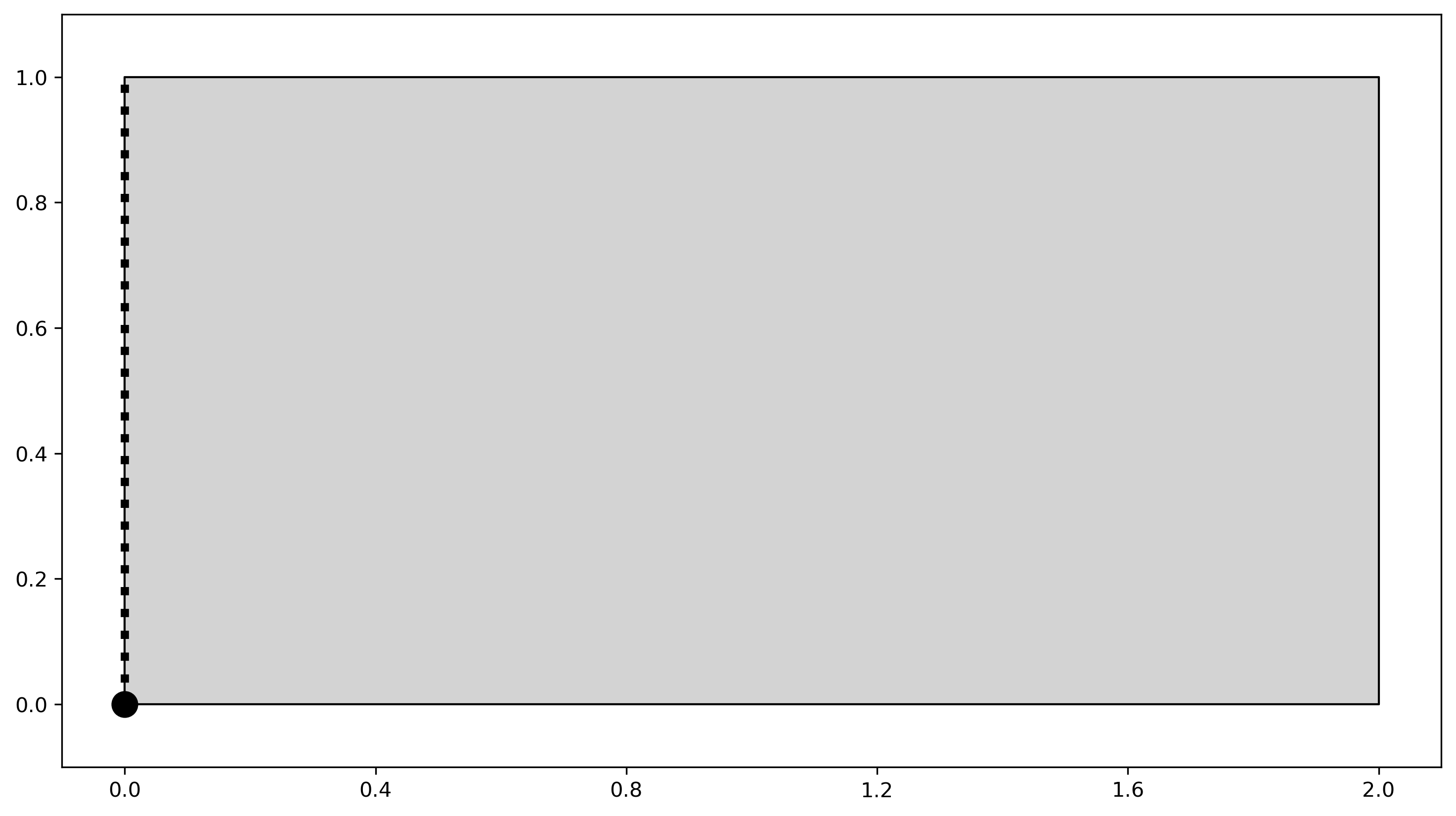}
        \caption{$F = \{(0,0)\},\;\cC = [0,2]\times [0,1], \; \bar \Phi(F,\cC)=1$.}
    \end{subfigure} 
\end{tabular}
\caption{Outer facial distance $\bar\Phi(F,\cC)$
equals the length of the dotted line(s).
}\label{fig:outerfacial}
\end{figure}

\bigskip

To derive sublinear convergence rates for FW variants, we will rely on the following lemma from~\cite{polyak1987,borwein2014analysis}:
\begin{lemma}\label{lemma:borwein}
    Suppose that $p > 0$ and $(\beta_t)_{t\in\N}, (\sigma_t)_{t\in\N}$ are such that $\beta_t, \sigma_t \geq 0$ and $\beta_{t+1}\leq (1-\sigma_t\beta_t^p)\beta_t$ for $t \in \N$. Then, for all $t\in \N$, it holds that
    \begin{align*}
        \beta_t \leq \left(\beta_0^{-p} + p \sum_{i=0}^{t-1}\sigma_i \right)^{-\frac{1}{p}}.
    \end{align*}
\end{lemma}

\section{Algorithms}\label{sec.algorithms}

In this section, we present affine-invariant convergence rates for the vanilla Frank-Wolfe algorithm (FW) in Section~\ref{sec.vanilla}, the two active set variants away-step Frank-Wolfe algorithm (AFW) and blended-pairwise Frank-Wolfe algorithm (BPFW) in Section~\ref{sec.away_blended}, and the in-face Frank-Wolfe algorithm (IFW) in Section~\ref{sec.in_face}.  Although our discussion focuses on the case when $\cC\subseteq \R^n$ is a polytope, we note that our developments for both FW and IFW apply in the more general context when $\cC\subseteq \R^n$ is a compact convex set.

\subsection{Vanilla Frank-Wolfe algorithm.}\label{sec.vanilla}
\begin{algorithm}[!t]
    \caption{Frank-Wolfe algorithm (FW)}\label{alg:vfw}
	\begin{algorithmic}        
    \State{\bf Input:} {$\xx^{(0)}\in\cC$.
    }
    \State{\bf Output:} {$\xx^{(t)}\in\cC$ for $t\in\N$.}
    \hrulealg
    \For{$t\in\N$}
    \State {$\vv^{(t)} \gets \argmin_{\vv\in \cC} \langle \nabla f(\xx^{(t)}), \vv \rangle$ }
    \State {$\xx^{(t+1)} \gets \xx^{(t)} + \eta_t (\vv^{(t)}- \xx^{(t)})$
    for some $\eta_t \in [0,1]$}
    \EndFor 
\end{algorithmic}    
\end{algorithm}
As a preamble to our main developments, we first discuss the conceptually more approachable case of the vanilla Frank-Wolfe algorithm as described in Algorithm~\ref{alg:vfw}. This case provides a blueprint for our subsequent discussion of Frank-Wolfe variants.

The \emph{radial distance} introduced by Gutman and Pe\~na~\cite{gutman2021condition} is the key ingredient to formalize the affine-invariant convergence rate of FW. 

\begin{definition}[Radial distance]\label{def.radial_distance}
    Let $\cC\subseteq \R^n$ be a polytope. Define the \emph{radial distance} $\rad\colon \cC\times\cC\to[0,1]$ as follows. For $\yy,\xx\in \cC$ let the radial distance between $\yy$ and $\xx$ be
\[
        \rad(\yy,\xx)  = \min\{\rho \geq 0 \mid \yy-\xx = \rho (\vv-\xx)  \ \text{for some } \vv\in \cC\}.
\]
\end{definition}
Note that $\rad(\yy,\xx)= 0$ if and only if $\yy = \xx$.  Figure~\ref{fig:radial} displays visualizations of the function $\rad(\cdot,\xx)$ on $\cC$ for $\cC=[0,1]\times [0,1],\; \xx = (0.2,0.2)$ and also for $\cC=[0,2]\times [0,1],\; \xx = (0.4,0.2)$.  Figure~\ref{fig:radial}(b) is a stretched version of Figure~\ref{fig:radial}(a) due of the affine invariance of $\rad$.  

\begin{figure}[!t]
\captionsetup[subfigure]{justification=centering}
 \hspace{-.6in}
\begin{tabular}{c c} 
\begin{subfigure}{.55\textwidth}
    \centering
        \includegraphics[width=.62\textwidth]{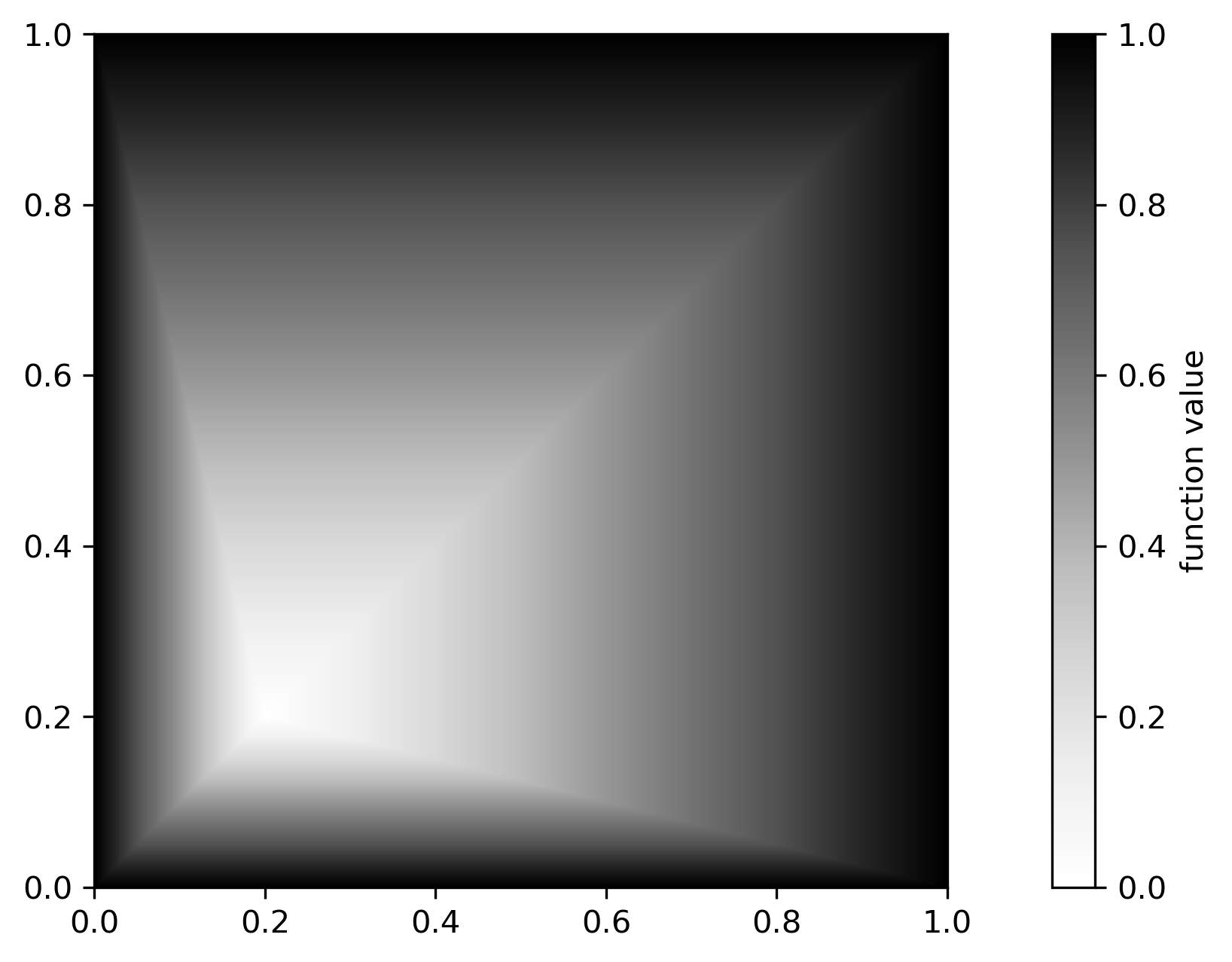}
        \caption{$\cC = [0,1]\times [0,1], \xx = (0.2,0.2)$.}
    \end{subfigure} & \hspace{-.5in}
    \begin{subfigure}{.575\textwidth}
    \centering
        \includegraphics[width=1\textwidth]{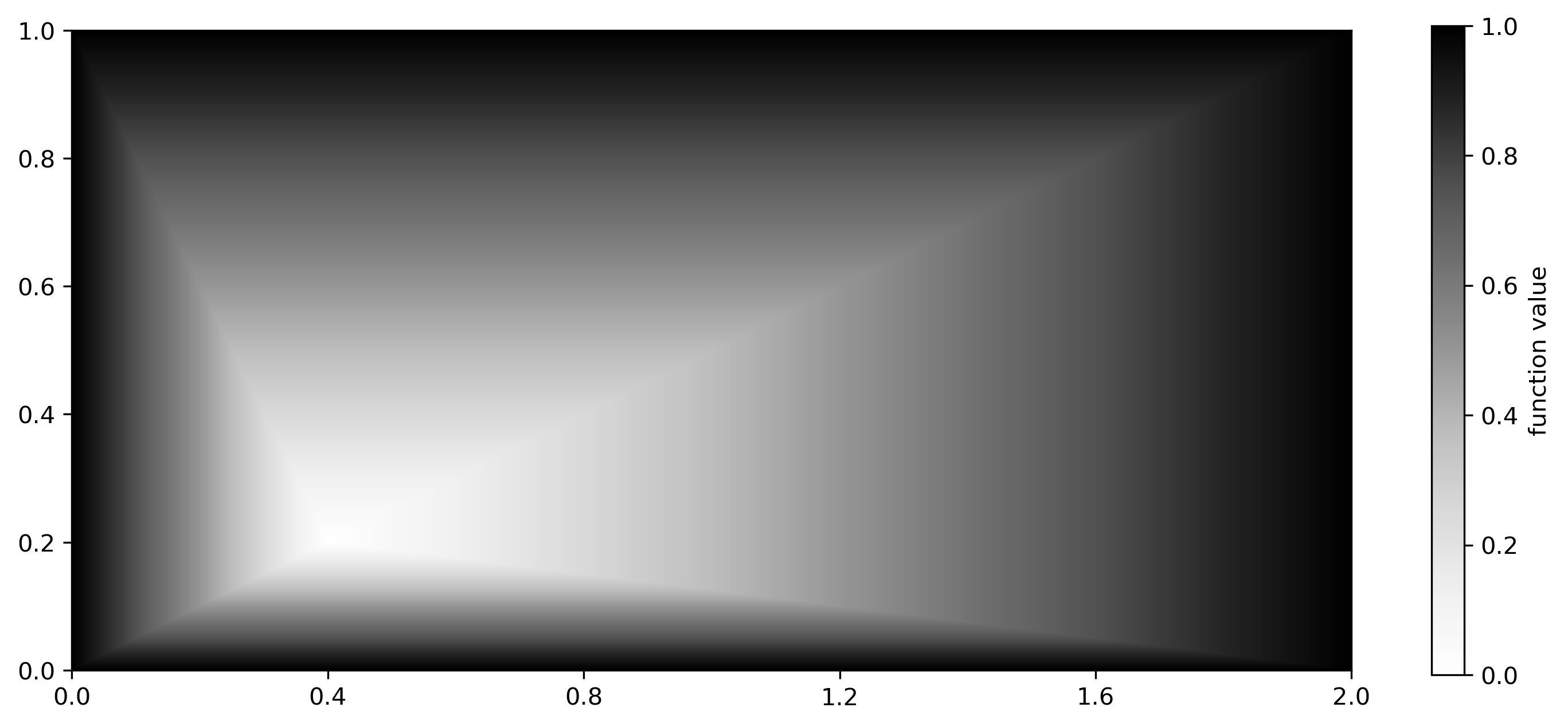}
        \caption{$\cC = [0,2]\times [0,1], \xx = (0.4,0.2)$.}
    \end{subfigure} 
\end{tabular}
\caption{Visualization of the function $\rad(\cdot,\xx)$ on $\cC$ for the radial distance $\rad$ of $\cC$.  The shade of gray indicates the function value. Points $\yy\in \cC$ with $\rad(\yy,\xx) = 0$ are in white and points $\yy\in \cC$ with $\rad(\yy,\xx)=1$ are in black.}\label{fig:radial}
\end{figure}

Functions of the form $\rad(\cdot,\xx)$ are somewhat intuitive and easy to visualize as they correspond to the gauge function of the set $\cC-\xx$.  However, for our analysis functions of the form $\rad(\yy,\cdot)$ are more relevant.  Figure~\ref{fig:radialrev} displays visualizations of the function $\rad(\yy,\cdot)$ on $\cC$ for $\cC=[0,1]\times [0,1]$ and various choices of $\yy\in \cC$.  Figure~\ref{fig:radialrev}(c) shows that when $\yy$ is a vertex of $\cC$, it holds that $\rad(\yy,\xx) = 1$ for every $\xx\in \cC$ with $\xx\ne \yy$.  Figure~\ref{fig:radialrev}(a) and 
Figure~\ref{fig:radialrev}(b) show that a similar phenomenon occurs as $\yy$ approaches a vertex of $\cC$.  Similar phenomena occur when $\yy$ lies on, or approaches, the relative boundary of $\cC$.

\begin{figure}[ht]
\captionsetup[subfigure]{justification=centering}
 \hspace{-.6in}
\begin{tabular}{ccc} 
\begin{subfigure}{.5\textwidth}
    \centering
        \includegraphics[width=.6\textwidth]{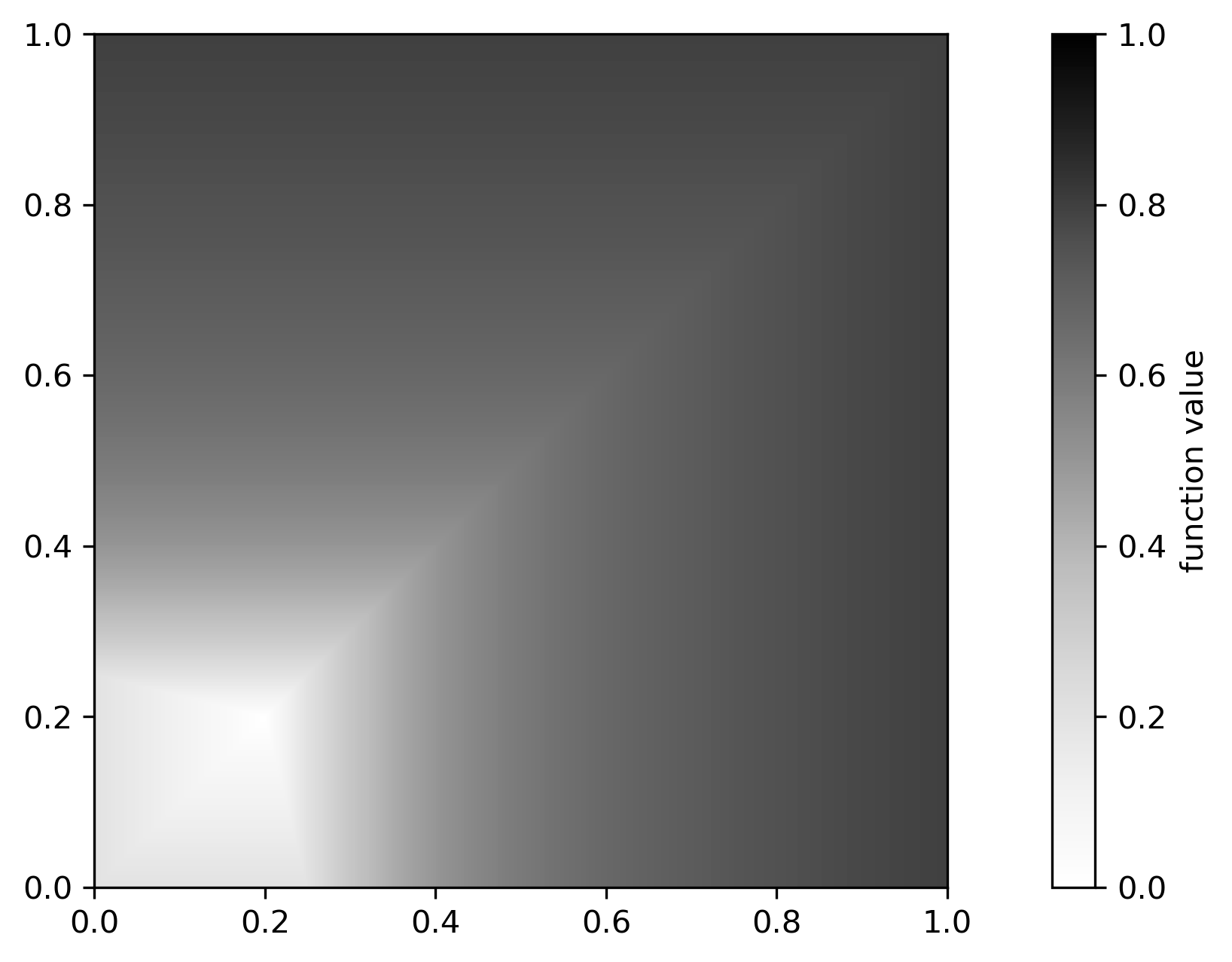}
        \caption{$\cC = [0,1]\times [0,1], \yy = (0.2,0.2)$.}
    \end{subfigure} & \hspace{-.5in}
 \hspace{-.7in}    
    \begin{subfigure}{.5\textwidth}
    \centering
        \includegraphics[width=.6\textwidth]{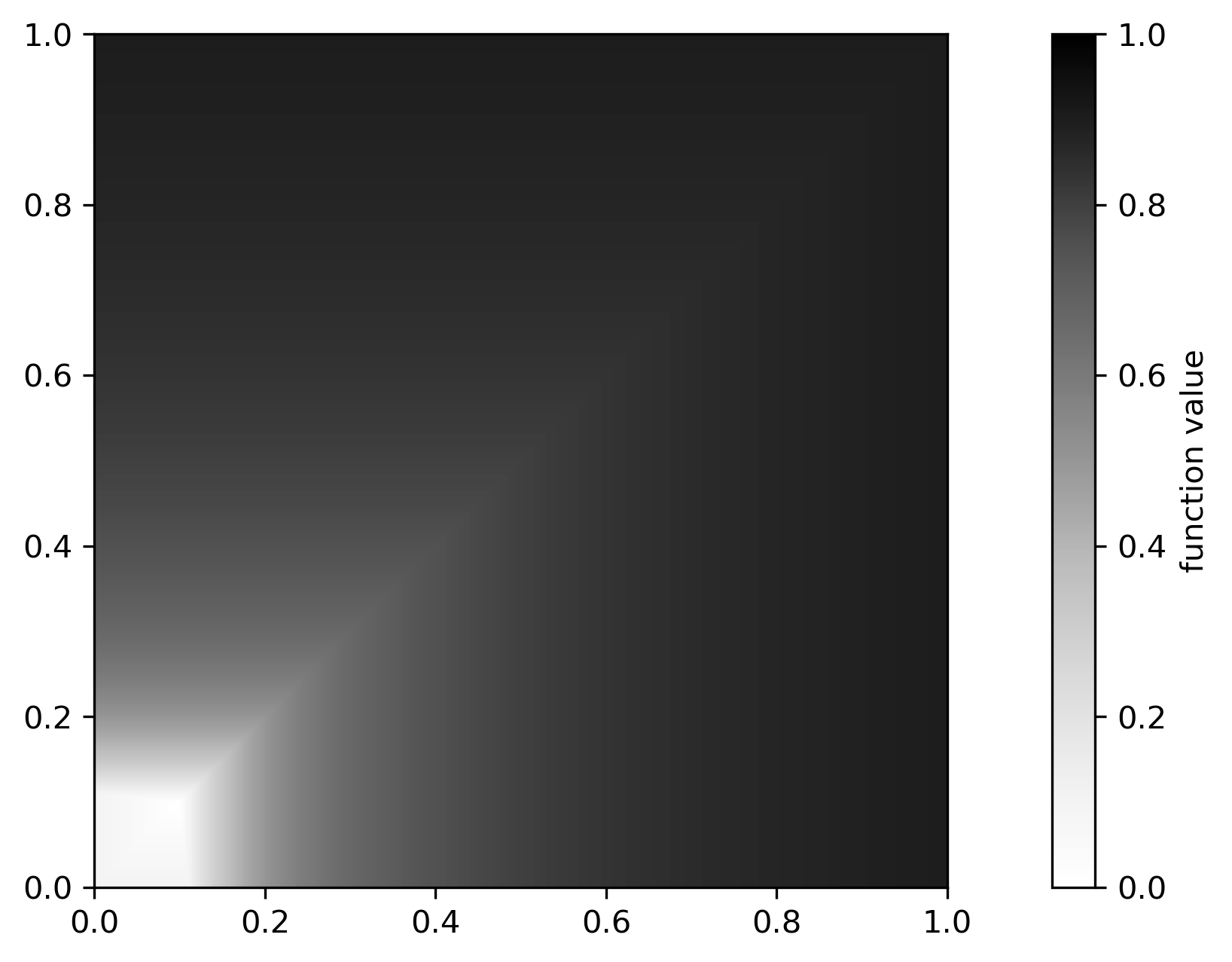}
        \caption{$\cC = [0,1]\times [0,1], \yy = (0.1,0.1)$.}
    \end{subfigure} 
 \hspace{-1.05in}    
    \begin{subfigure}{.5\textwidth}
    \centering
        \includegraphics[width=.6\textwidth]{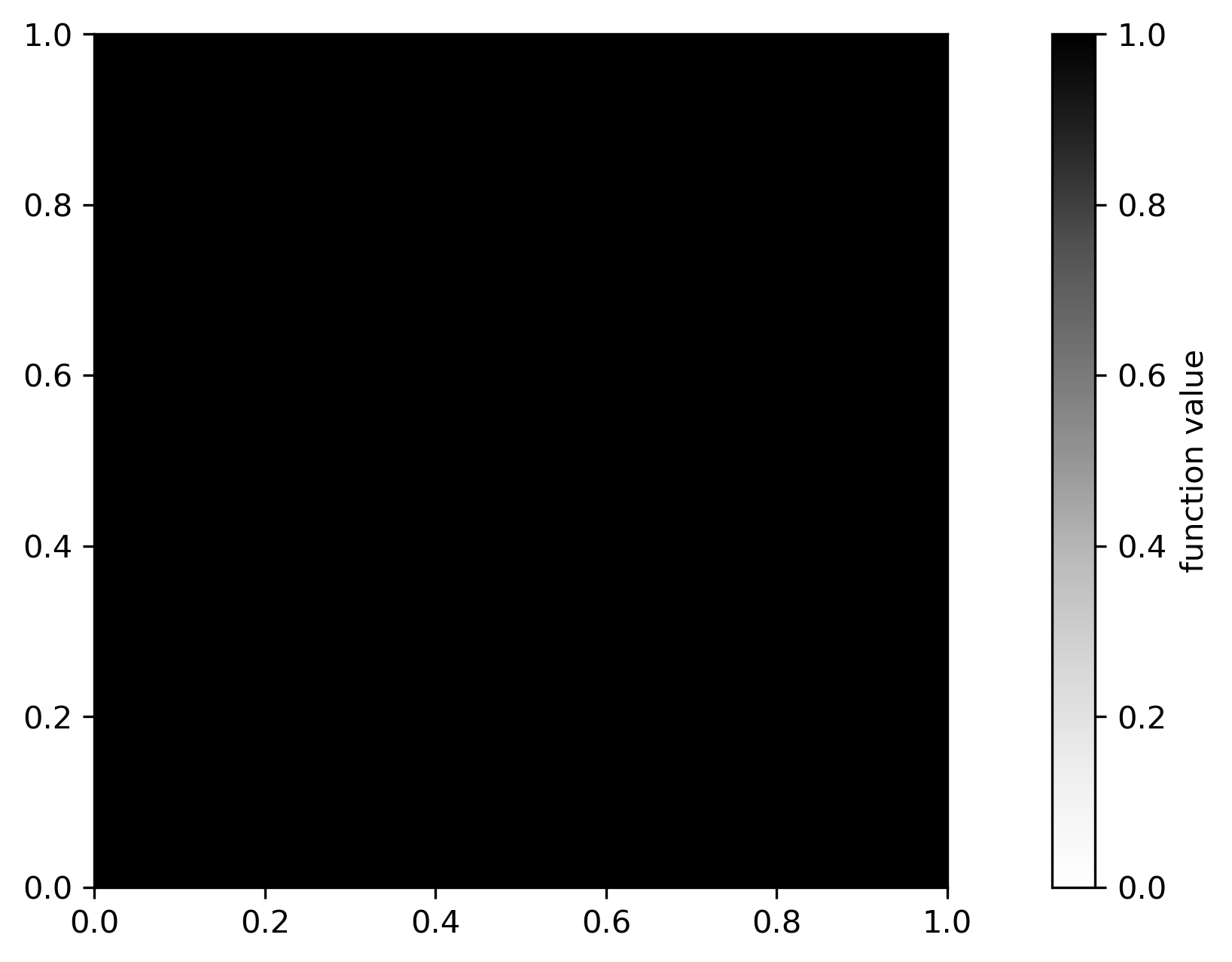}
        \caption{$\cC = [0,1]\times [0,1], \yy = (0,0)$.}
    \end{subfigure} 
\end{tabular}
\caption{Visualization of the function
 $\rad(\yy,\cdot)$ on $\cC$ for the radial distance 
$\rad$ of $\cC$.  The shade of gray indicates the function value. Points $\xx\in \cC$ with $\rad(\yy,\xx) = 0$ are in white and points $\xx\in \cC$ with $\rad(\yy,\xx)=1$ are in black.}
\label{fig:radialrev}
\end{figure}

As detailed in Theorem~\ref{thm.vanilla} below, the convergence rate of Algorithm~\ref{alg:vfw} follows from the curvature of $(\cC,f)$ and error bound of $(\cC,f)$ relative to $\rad$.
Suppose that $(\cC,f)$ has  $L$-curvature. We can minimize the left-hand-side of \eqref{eq.step_size_motivator} for $\dd := \vv^{(t)}- \xx^{(t)}$.  This step-size rule is referred to as \emph{line-search}:
\[
    \eta_t = \argmin_{\eta\in [0,1]} f(\xx^{(t)}+\eta (\vv^{(t)}- \xx^{(t)})).
\]
Alternatively, note that the right-hand-side of \eqref{eq.step_size_motivator} is a majorant for $f(\xx+\eta\dd)$. Thus, another natural step-size rule minimizes the right-hand-side of \eqref{eq.step_size_motivator} over $\eta\in [0,1]$. The resulting majorant-minimization step is referred to as \emph{short-step}:
\[
    \eta_t = \argmin_{\eta\in [0,1]}\left\{f(\xx^{(t)}) + \eta \ip{ \nabla f(\xx^{(t)})}{\vv^{(t)} - \xx^{(t)}} + \frac{L\eta^2}{2}\right\}.
\]
For either of the line-search or short-step rules, the following progress bound holds 
when $\eta_t < 1$
\begin{align}\label{eq.short.step}
    f(\xx^{(t+1)}) \le f(\xx^{(t)}) - \frac{\ip{\nabla f(\xx^{(t)})}{\xx^{(t)}-\vv^{(t)}}^2}{2L},
\end{align}
and the following progress bound holds 
when $\eta_t = 1$
\begin{align}\label{eq.short.step.2}
    f(\xx^{(t+1)}) \le f(\xx^{(t)}) - \frac{\ip{\nabla f(\xx^{(t)})}{\xx^{(t)}-\vv^{(t)}}}{2}.
\end{align}

Our convergence analysis of FW relies on the progress bounds~\eqref{eq.short.step} and \eqref{eq.short.step.2}, and thus yields identical convergence rates for line-search and short-step.

\begin{theorem}[FW]\label{thm.vanilla}
    Let $\cC \subseteq \R^n$ be a polytope, let $f\colon \cC \to \R$ be convex and differentiable in an open set containing $\cC$, and suppose that $(\cC,f)$ has $L$-curvature for some $L >0$, and satisfies the $(\mu,\theta)$-error bound relative to $\rad$ for some $\mu > 0$ and $\theta \in[0,1/2]$. Then the iterates of Algorithm~\ref{alg:vfw} with line-search or short-step satisfy the following convergence rates.
    \begin{itemize}
    \item     When $\theta = 1/2$, for all $t\in\N$ the following linear convergence rate holds:
    \begin{equation}\label{eq.linear.vfw}
        f(\xx^{(t)})-f^* \leq (f(\xx^{(0)}) -f^*)\left(1-\min\left\{\frac{\mu}{2L},\frac{1}{2}\right\}\right)^{t}.
    \end{equation}
    \item 
    When $\theta \in [0,1/2)$ the following initial linear convergence rate holds for $t\in \{0,1,\dots,t_0\}$
    \begin{equation}\label{eq.linearregime.vfw}
        f(\xx^{(t)})-f^* \leq \frac{f(\xx^{(0)}) -f^*}{2^t},
    \end{equation}
 where $t_0\in\N$ is the smallest $t$ such that
\[
        \frac{\mu^{2\theta}(f(\xx^{(t)}) - f^*)^{1-2\theta}}{2L} \le \frac{1}{2}.
\]
Then for $t\in \N$ with $t >  t_0$, the following sublinear convergence rate holds:
    \begin{equation}\label{eq.sublinear.vfw}
        f(\xx^{(t)})-f^*  \leq\left((f(\xx^{(t_0)}) -f^*)^{(2\theta-1)} + \frac{(1-2\theta)\mu^{2\theta}(t-t_0)}{2L}\right)^{\frac{1}{2\theta-1}}.
    \end{equation}
    \end{itemize}
\end{theorem}

The rates in Theorem~\ref{thm.vanilla} interpolate between the iconic rate $\cO(t^{-1})$ \cite{jaggi2013revisiting} when $\theta = 0$ and the linear rate~\eqref{eq.linear.vfw} when $\theta = 1/2$.
The proof of Theorem~\ref{thm.vanilla} relies on the following \emph{scaling lemma.}

\begin{lemma}\label{lem.scaling.radial}
    Let $\cC \subseteq \R^n$ be a polytope  and let $f\colon \cC \to \R$ be convex and differentiable in an open set containing $\cC$. Then, for all $\xx\in \cC\setminus X^*$, it holds that
    \begin{equation}\label{eq.scaling.rad}
        \max_{\vv\in \cC} \ip{\nabla f(\xx)}{\xx-\vv} \geq \frac{f(\xx) - f^*}{\rad(X^*,\xx)}.
    \end{equation}
    In particular, if $(\cC,f)$ satisfies a $(\mu,\theta)$-error bound relative to $\rad$ for some $\mu > 0$ and $\theta \in [0, 1/2]$, then for all $\xx\in \cC$ it holds that
    \begin{equation}\label{eq.decrease}
        \max_{\vv\in \cC} \ip{\nabla f(\xx)}{\xx-\vv} \geq \mu^\theta(f(\xx) - f^*)^{1-\theta}.
    \end{equation}
\end{lemma}
\begin{proof}{Proof.} Suppose that $\xx\in \cC\setminus X^*$. Let $\xx^*\in X^*$ be such that $\rad(\xx^*,\xx) = \rad(X^*,\xx)$. The definition of $\rad$ implies that there exists $\ww\in\cC$ such that $\xx^*-\xx = \rad(X^*,\xx)\cdot(\ww-\xx)$. Thus,
\[
        \max_{\vv\in \cC} \ip{\nabla f(\xx)}{\xx-\vv} \ge \ip{\nabla f(\xx)}{\xx-\ww} = \frac{\ip{\nabla f(\xx)}{\xx-\xx^*}}{\rad(X^*,\xx)}  \ge  \frac{f(\xx) - f^*}{\rad(X^*,\xx)},
\]
    where the last inequality follows from convexity of $f$.  Thus~\eqref{eq.scaling.rad} follows.  To show~\eqref{eq.decrease}, suppose $\xx \in \cC$.  If $\xx\in X^*$ then~\eqref{eq.decrease}  trivially holds.  Otherwise, when $\xx \in \cC\setminus X^*$, observe that~\eqref{eq.scaling.rad} and the $(\mu,\theta)$-radial error bound imply that
\[
        \max_{\vv\in \cC} \ip{\nabla f(\xx)}{\xx-\vv}  \ge  \frac{f(\xx) - f^*}{\rad(X^*,\xx)}   =
        \frac{(f(\xx) - f^*)^{1-\theta}(f(\xx) - f^*)^\theta}{\rad(X^*,\xx)} \ge \mu^\theta(f(\xx) - f^*)^{1-\theta}.
\]
\end{proof}
\begin{proof}{Proof of Theorem~\ref{thm.vanilla}.}
    The $L$-curvature and the choice of step-size imply that for $t\in\N$ either~\eqref{eq.short.step} holds when $\eta_t<1$ or~\eqref{eq.short.step.2} holds when $\eta_t=1$. Thus, Lemma~\ref{lem.scaling.radial} implies that for each $t\in\N$ one of the following two bounds holds. If $\eta_t <1$ then
    \[
        f(\xx^{(t+1)})  \le f(\xx^{(t)}) - \frac{\ip{\nabla f(\xx^{(t)})}{\xx^{(t)}-\vv^{(t)}}^2}{2L}   
                        \le f(\xx^{(t)}) - \frac{\mu^{2\theta}(f(\xx^{(t)}) - f^*)^{2(1-\theta)}}{2 L},
    \]
and if $\eta_t=1$ then
    \[
    f(\xx^{(t+1)})  \le f(\xx^{(t)}) - \frac{\ip{\nabla f(\xx^{(t)})}{\xx^{(t)}-\vv^{(t)}}}{2} \le  f(\xx^{(t)}) - \frac{f(\xx^{(t)})-f^*}{2}.
    \]    
    Hence for all $t\in \N$,
              \begin{equation}\label{eq.decrease.vfw}
        f(\xx^{(t+1)})-f^* \le (f(\xx^{(t)})-f^*)\left(1-\min\left\{\frac{\mu^{2\theta}}{2L}(f(\xx^{(t)})-f^*)^{1-2\theta},\frac{1}{2}\right\}\right) .
    \end{equation}
 Thus~\eqref{eq.linear.vfw} readily follows when $\theta = 1/2$.  On the other hand, when $\theta\in[0,1/2)$, Inequality~\eqref{eq.decrease.vfw} yields~\eqref{eq.linearregime.vfw}
 for $t=0,1,\dots,t_0$. 
Finally for $t\in \N$ with $t>t_0$, Inequality~\eqref{eq.sublinear.vfw} follows from~\eqref{eq.decrease.vfw} by applying
    Lemma~\ref{lemma:borwein} with $p=1-2\theta$, $\sigma_t= \frac{\mu^{2\theta}}{2L}$, and $\beta_t = f(\xx^{(t)})-f^*$.
\end{proof}

We next discuss common sufficient conditions for the assumptions of Theorem~\ref{thm.vanilla} to hold.  First, as noted in Section~\ref{sec.prelims}, the pair $(\cC,f)$ has $L$-curvature for some $L > 0$ whenever $f$ is smooth with respect to a norm on $\cC$.  Let $\relint(\cC)$ and $\relbnd(\cC)$ denote respectively the relative interior and relative boundary of $\cC$. 
The following proposition shows that when $X^*\subseteq\relint(\cC)$, an error bound relative to $\rad$ holds provided $f$ satisfies a Hölderian error bound over $\cC$ relative to some norm. 

\begin{proposition}[Radial error bound]\label{prop:int_aff_err}
Let $\R^n$ be endowed with a norm $\|\cdot\|$, let $\cC \subseteq \R^n$ be a polytope, and let $f\colon \cC \to \R$ be convex.
Suppose that $X^* \subseteq \relint(\cC)$ and $f$ satisfies a $(\mu, \theta)$-Hölderian error bound over $\cC$ with respect to the norm $\|\cdot\|$ for some $\mu>0$ and $\theta \in(0,1/2]$. Then $(\cC,f)$ satisfies a $(\tilde \mu, \theta)$-error bound relative to $\rad$ for 
    \[
    \tilde \mu = \mu\cdot \dist(X^*,\relbnd(\cC))^{1/\theta}.
    \]
\end{proposition}
\begin{proof}{Proof.}
The construction of $\rad$ implies that for all $\xx^*\in X^*$ and $\xx\in \cC$
\[
\|\xx^*-\xx\| = \rad(\xx^*,\xx) \cdot \|\xx-\uu\| \ge \rad(\xx^*,\xx) \cdot \|\xx^*-\uu\|
\]
 for some $\uu \in \relbnd(\cC)$. Therefore for all $\xx\in \cC$
\[
\min_{\xx^*\in X^*} \|\xx^*-\xx\| \ge \rad(X^*,\xx) \cdot  \dist(X^*,\relbnd(\cC)).
\]
Since $f$ satisfies a $(\mu, \theta)$-Hölderian error bound over $\cC$ with respect to the norm $\|\cdot\|$, it follows that for all $\xx\in \cC$
\[
\left(\frac{f(\xx)-f^*}{\mu}\right)^\theta \geq \min_{\xx^*\in X^*} \|\xx^*-\xx\| \ge \rad(X^*,\xx) \cdot  \dist(X^*,\relbnd(\cC)),
\]
and thus 
\[
\left(\frac{f(\xx)-f^*}{\tilde \mu}\right)^\theta \geq \rad(X^*,\xx).
\]
\end{proof}

The lower bound in Proposition~\ref{prop:int_aff_err} is zero when $X^* \cap \relbnd(\cC) \ne \emptyset$. Thus in this case Theorem~\ref{thm.vanilla} fails to show linear convergence even when $f$ is smooth and $(\cC,f)$ satisfy the $(\mu,1/2)$ H\"olderian error bound relative to some norm in $\R^n$.
 This is consistent with the shortcoming of FW discovered and documented by Wolfe~\cite{wolfe1970convergence}.
Indeed, Wolfe~\cite{wolfe1970convergence} characterizes instances of \eqref{eq.opt} where $\cC$ is a polytope and $f$ is smooth and strongly convex for which FW with line-search or short-step cannot converge faster than $\Omega(t^{-1-\epsilon})$ for any $\epsilon > 0$. This phenomenon, referred to as \emph{Wolfe's lower bound}, was the core motivation behind the study of FW variants to overcome the limitations of the vanilla version.

\subsection{Away-step Frank-Wolfe algorithm and blended pairwise Frank-Wolfe algorithm}\label{sec.away_blended}

\begin{algorithm}[t]
    \caption{Away-step Frank-Wolfe algorithm (AFW) and blended pairwise Frank-Wolfe algorithm (BPFW)}\label{alg:tfw.full}
\begin{algorithmic}    
    \State {\bf Input:} {$\xx^{(0)} \in \vertices(\cC)$, that is, $ \xx^{(0)} =\cV\llambda^{(0)}\in\cC$ for some $\llambda^{(0)}\in\Delta(\cV)$ with $|S(\llambda^{(0)}|=1.$
    }
    \State{\bf Output:} {$\xx^{(t)}= \cV\llambda^{(t)}\in\cC$ for some $\llambda^{(t)}\in\Delta(\cV)$ for $t\in\N$.}
    \hrulealg
    \For{$t\in\N$}
    \State{compute $\dd^{(t)}_{\text{FW}}, \dd^{(t)}_{\text{AFW}}, \; \dd^{(t)}_{\text{BPFW}}$ as in~\eqref{eq.directions} for $\xx = \xx^{(t)}$ and $S = S(\llambda^{(t)})$}
    \State{select either $\dd^{(t)}\in \{\dd^{(t)}_{\text{FW}}, \dd^{(t)}_{\text{AFW}}\}$ in AFW or $\dd^{(t)}\in \{\dd^{(t)}_{\text{FW}},\; \dd^{(t)}_{\text{BPFW}}\}$ in BPFW that minimizes $\ip{\nabla f(\xx^{(t)})}{\dd^{(t)}}$}\label{line.tfw.full.selection}
    \State{compute $\eta_{t,\max}$ as in~\eqref{eq.max.step-size} for $\llambda = \llambda^{(t)}$} 
    \State{$\xx^{(t+1)} \gets \xx^{(t)} + \eta_t \dd^{(t)}$ for some $\eta_t \in [0,\eta_{t,\max}]$}
    \State{update $\llambda^{(t+1)}\in\Delta(\cV)$ as in~\eqref{eq.update.S.alpha} for $\llambda = \llambda^{(t)}$}
	\EndFor
\end{algorithmic}    
\end{algorithm}

The main update in the vanilla Frank-Wolfe algorithm can be written as
\[
    \xx^{(t+1)} \gets \xx^{(t)} + \eta_t \dd^{(t)}
\]
where $\dd^{(t)} = \vv^{(t)}- \xx^{(t)}$ is the \emph{Frank-Wolfe direction} defined via $\vv^{(t)} = \argmin_{\vv\in \cC} \langle \nabla f(\xx^{(t)}), \vv \rangle$.
As noted above, FW, limited to FW directions, does not admit fast convergence rates when there are optimal solutions on the relative boundary of $\cC$.
To overcome this limitation, FW variants have been studied that admit more flexible choices of $\dd^{(t)}$. In this section, we focus on the away-step Frank-Wolfe algorithm (AFW) \cite{beck2017linearly,wolfe1970convergence, lacoste2015global} and the blended pairwise Frank-Wolfe algorithm (BPFW) \cite{tsuji2022pairwise}. To facilitate a richer choice of search directions, both  variants maintain and rely on a {\em vertex representation} of $\xx^{(t)}$ as we next explain. 

Let $\cC \subseteq \R^n$ be a polytope with vertex set $\cV := \vertices(\cC)$ and let
\[
    \Delta(\cV):=
    \left\{\llambda = (\lambda_{\vv})_{\vv\in \cV} \mid \lambda_{\vv} \ge 0  \ \text{for } \vv \in \cV \ \text{and} \sum_{\vv\in \cV} \lambda_{\vv} =1\right\}.
\]
Any $\xx\in \cC$ has a \emph{vertex representation} of the form $\xx = \sum_{\vv\in \cV} \lambda_{\vv} \vv$ for some $\llambda \in \Delta(\cV)$.
For ease of notation, we write $\cV \llambda$ for $\sum_{\vv\in \cV} \lambda_{\vv} \vv$.  For $\llambda \in \Delta(\cV)$, the \emph{support set} of $\llambda$  is defined as:
\[    S(\llambda) = \{\vv\in \cV \mid \lambda_{\vv} > 0\}.
\]
For $\xx\in \cC$, we further define
\[
    \SS(\xx) = \{S(\llambda) \mid \xx = \cV\llambda  \ \text{for some } \llambda \in \Delta(\cV)\}.
\]
Suppose that $\xx\in \cC$ and $\xx = \cV\llambda$ for some 
$\llambda\in\Delta(\cV)$.  Let $S = S(\llambda)$ and define the \emph{away vertex} $\aa\in S$, the \emph{local FW vertex} $\zz \in S$, and the \emph{global FW vertex} $\vv\in \cV$ as follows:
\[
    \aa  = \argmax_{\yy\in S} \ip{\nabla f(\xx)}{\yy}, \qquad \zz = \argmin_{\yy\in S} \ip{\nabla f(\xx)}{\yy}, \qquad
    \vv  = \argmin_{\yy\in \cC} \ip{\nabla f(\xx)}{\yy}.
\]
The above away, local FW, and global FW vertices in turn yield the following \emph{FW}, \emph{away}, and \emph{blended pairwise} directions:
\begin{equation}\label{eq.directions}
    \dd_{\text{FW}} = \vv-\xx, \qquad \dd_{\text{AFW}} = \xx-\aa, \qquad   \dd_{\text{BPFW}} = \zz - \aa.
\end{equation}
Notice that the away and blended pairwise directions reduce the weight of a vertex in the vertex representation of $\xx$ and increase the weight of other vertices.  In contrast, the FW direction increases the weight of the vertex $\vv$ and reduces the weight of all other vertices equally.
For $\dd \in \{\dd_{\text{FW}}, \dd_{\text{AFW}}, \dd_{\text{BPFW}}\}$, there is a natural maximum step-size $\eta_{\max}$ that ensures that $\xx+\eta \dd \in \cC$ for $\eta\in [0,\eta_{\max}]$.  The precise calculation of $\eta_{\max}$ requires an explicit vertex representation of $\xx$  of the form $\xx = \cV\llambda$ such that $S(\llambda) = S$.
For each of the directions in~\eqref{eq.directions}, the maximum step-size $\eta_{\max}$ is as follows, see \cite{wolfe1970convergence,lacoste2015global,tsuji2022pairwise}:
\begin{equation}\label{eq.max.step-size}
    \eta_{\max} = \left\{\begin{array}{rl}
        1 ,                                    & \text{if} \ \dd=\dd_{\text{FW}} = \vv-\xx    \\
        \frac{\lambda_{\aa}}{1-\lambda_{\aa}}, & \text{if} \ \dd =\dd_{\text{AFW}}= \xx-\aa   \\
        \lambda_{\aa}           ,              & \text{if} \ \dd =\dd_{\text{BPFW}}= \zz-\aa.
    \end{array} \right.
\end{equation}
Furthermore, for $\eta\in [0,\eta_{\max}]$ the point $\xx^+ = \xx+\eta\dd$ has a vertex representation $\xx^+ = \cV\llambda^+$ where
\begin{equation}\label{eq.update.S.alpha}
    \lambda^{+}_{\ss} = \left\{\begin{array}{rl} (1-\eta)\lambda_{\ss}+\eta, & \ \text{if} \ \dd=\dd_{\text{FW}} = \vv-\xx \ \text{and} \ \ss=\vv                  \\
        (1-\eta)\lambda_{\ss} ,          & \ \text{if} \ \dd=\dd_{\text{FW}} = \vv-\xx \ \text{and} \ \ss\ne\vv                \\ [1ex]
        (1+\eta)\lambda_{\ss}-\eta   ,   & \ \text{if} \ \dd=\dd_{\text{AFW}} = \xx-\aa \ \text{and} \ \ss=\aa                 \\
        (1+\eta)\lambda_{\ss}     ,      & \ \text{if} \ \dd=\dd_{\text{AFW}} = \xx-\aa \ \text{and} \ \ss\ne\aa               \\[1ex]
        \lambda_{\ss}-\eta     ,         & \ \text{if} \ \dd=\dd_{\text{BPFW}} = \zz-\aa \ \text{and} \ \ss=\aa                \\
        \lambda_{\ss}+\eta    ,          & \ \text{if} \ \dd=\dd_{\text{BPFW}} = \zz-\aa \ \text{and} \ \ss=\zz                \\
        \lambda_{\ss}         ,          & \ \text{if} \ \dd=\dd_{\text{BPFW}} = \zz-\aa \ \text{and} \ \ss\not\in\{\aa,\zz\}.
    \end{array} \right.
\end{equation}

We next establish affine-invariant convergence results for  Algorithm~\ref{alg:tfw.full} analogous to Theorem~\ref{thm.vanilla}.  To do so, we rely on the following \emph{vertex distance} in lieu of the radial distance.

\begin{definition}[Vertex distance]\label{def.vertex_distance}
    Let $\cC\subseteq \R^n$ be a polytope.  Define the \emph{vertex distance} $\v\colon \cC\times\cC\to[0,1]$ as follows. For $\yy,\xx \in\cC$ let the vertex distance between $\yy$ and $\xx$ be
\[
          \v(\yy,\xx)                                                                                                                             
          = \max_{S\in\SS(\xx)} \min\{\gamma \geq 0 \mid  \yy-\xx = \gamma (\vv-\uu) \;\text{for some}\;\vv\in C\;\text{and}\;\uu\in\conv(S)\}.
\]
\end{definition}
Note that $\v(\yy, \xx)= 0$ if and only if $\yy = \xx$.  Observe that if $\cC\subseteq\R^n$ is a polytope, then for all $\yy,\xx\in \cC$, it holds that
$
    \v(\yy,\xx) \leq \rad(\yy,\xx).
$ 
In particular, if $(\cC,f)$ satisfies the $(\mu,\theta)$-error bound relative to $\rad$ then $(\cC,f)$ also satisfies the $(\mu,\theta)$-error bound relative to $\v$ and it would typically do so for a larger $\mu$.

Figure~\ref{fig:vertex} displays visualizations of the function $\v(\yy,\cdot)$ on $\cC$ for $\cC=[0,1]\times [0,1],\; \yy = (0,0)$ and also for $\cC=[0,2]\times [0,1],\; \yy = (0,0)$.
As in Figure~\ref{fig:radial}, Figure~\ref{fig:vertex}(b) is a stretched version of Figure~\ref{fig:vertex}(a) due of the affine invariance of $\v$.

\begin{figure}[!t]
\captionsetup[subfigure]{justification=centering}
 \hspace{-.6in}
\begin{tabular}{c c} 
\begin{subfigure}{.55\textwidth}
    \centering
        \includegraphics[width=.62\textwidth]{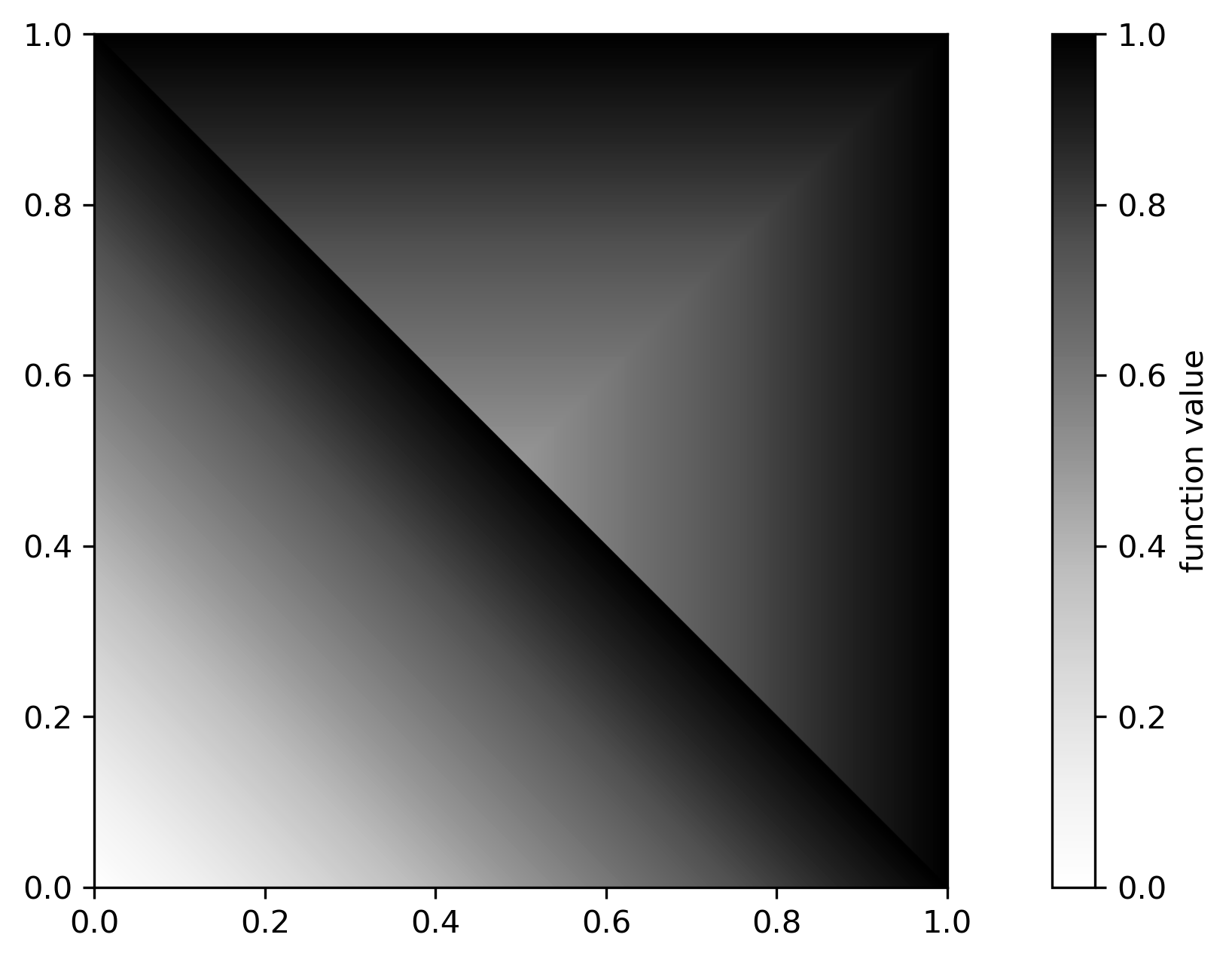}
        \caption{$\cC = [0,1]\times [0,1], \yy = (0,0)$.}
    \end{subfigure} & \hspace{-.5in}
    \begin{subfigure}{.575\textwidth}
    \centering
        \includegraphics[width=1\textwidth]{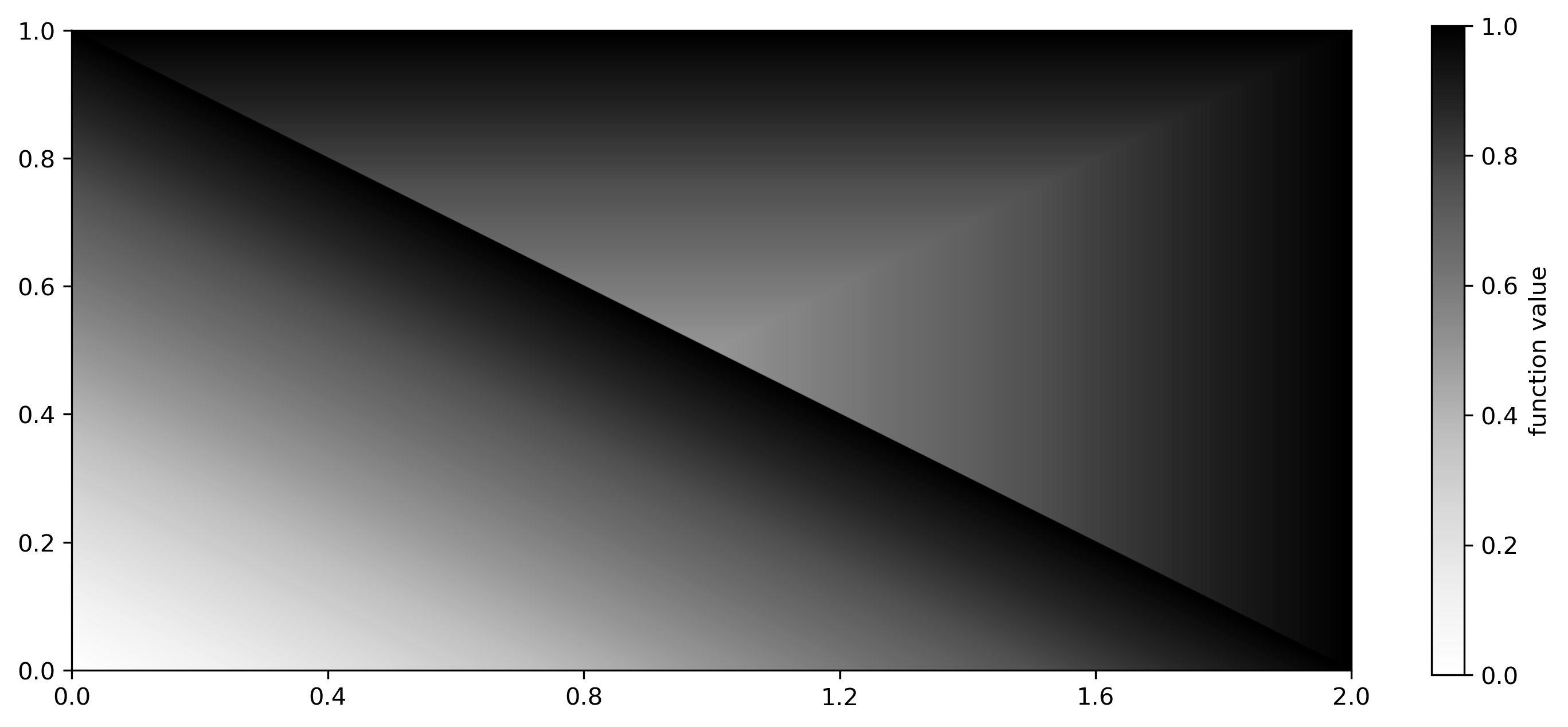}
        \caption{$\cC = [0,2]\times [0,1], \yy = (0,0)$.}
    \end{subfigure} 
\end{tabular}
\caption{
Visualization of the function
 $\v(\yy,\cdot)$ on $\cC$ for the vertex distance 
$\v$ of $\cC$.  The shade of gray indicates the function value. Points $\xx\in\cC$ with $\v(\yy,\xx) = 0$ are in white and points $\xx\in\cC$ with $\v(\yy,\xx)=1$ are in black.}\label{fig:vertex}
\end{figure}

As we detail below in Theorem~\ref{thm.variants}, the convergence of Algorithm~\ref{alg:tfw.full} is a consequence of the extended curvature  of $(\cC,f)$ and the error bound of $(\cC,f)$ relative to $\v$. 
Suppose that $(\cC,f)$ has extended $L$-curvature.  We consider two possible choices of step-size $\eta_t$ in Line~\ref{line.tfw.full.selection} of Algorithm~\ref{alg:tfw.full}.  The first one is \emph{line-search}:
\[
    \eta_t = \argmin_{\eta\in [0,\eta_{t,\max}]} f(\xx^{(t)}+\eta \dd^{(t)}).
\]
The second one is the \emph{short-step}:
\[
    \eta_t = \argmin_{\eta\in [0,\eta_{t,\max}]}\left\{f(\xx^{(t)}) + \eta \ip{ \nabla f(\xx^{(t)})}{\dd^{(t)}} + \frac{L\eta^2}{2}\right\}.
\]
It is easy to see that for either the line-search or short-step the following inequality holds when $\eta_t < \eta_{t,\max}$
\begin{equation}\label{eq.eta.loose}
    f(\xx^{(t+1)}) \le f(\xx^{(t)}) - \frac{\ip{\nabla f(\xx^{(t)})}{\dd^{(t)}}^2}{2L},
\end{equation}
and the following inequality holds when $\eta_t = \eta_{t,\max}\ge 1$
\begin{equation}\label{eq.eta.tight}
    f(\xx^{(t+1)}) \le f(\xx^{(t)}) + \frac{\ip{\nabla f(\xx^{(t)})}{\dd^{(t)}}}{2}.
\end{equation}

\begin{theorem}[AFW and BPFW]\label{thm.variants}
    Let $\cC \subseteq \R^n$ be a polytope, let $f\colon \cC \to \R$ be convex and differentiable in an open set containing $\cC$, and suppose that $(\cC,f)$ has extended $L$-curvature for some $L>0$,  and satisfies the $(\mu,\theta)$-error bound relative to $\v$ for some  $\mu > 0$ and $\theta\in[0,1/2]$.
    Then, the iterates of Algorithm~\ref{alg:tfw.full} with line-search or short-step satisfy the following convergence rates.
    
    \begin{itemize}
\item     When $\theta = 1/2$ the following linear convergence rate holds for all $t\in\N$:
    \begin{equation}\label{eq.linear.fwa}
        f(\xx^{(t)})-f^* \leq ( f(\xx^{(0)}) - f^* )\left(1-\min\left\{\frac{\mu}{8L}, \frac{1}{2}\right\}\right)^{\lceil t/2\rceil}.
    \end{equation}
\item     When $\theta \in [0,1/2)$, the following initial linear convergence rate holds for $t\in \{0,1,\dots,t_0\}$
    \begin{equation}\label{eq.linearregime.fwa}
        f(\xx^{(t)})-f^* \leq \frac{f(\xx^{(0)}) -f^*}{2^{\lceil t/2 \rceil}},
    \end{equation}
 where $t_0\in\N$ is the smallest $t$ such that
\[
        \frac{\mu^{2\theta}(f(\xx^{(t)}) - f^*)^{1-2\theta}}{8L} \le \frac{1}{2}.
\]
For $t\in \N$ with $t_0 < t \le t_1:= |S(\llambda^{(t_0)})|+t_0-1$ the following  holds:
\begin{equation}\label{eq.constantregime.fwa}
f(\xx^{(t)})-f^* \le f(\xx^{(t_0)}) -f^*.
\end{equation}
Finally, for $t\in \N$ with $t >  t_1$, the following sublinear convergence rate holds:
    \begin{equation}\label{eq.sublinearregime.fwa}
          f(\xx^{(t)})-f^* 
          \leq\left((f(\xx^{(t_1)}) -f^*)^{(2\theta-1)} + \frac{(1-2\theta)\mu^{2\theta}}{8L} \cdot \left\lceil \frac{t-t_1}{2} \right\rceil\right)^{\frac{1}{2\theta-1}}.
    \end{equation}
\end{itemize}
\end{theorem}

Once again, the rates in Theorem~\ref{thm.variants} interpolate between the iconic rate $\cO(t^{-1})$ when $\theta = 0$, to the linear rate~\eqref{eq.linear.fwa} when $\theta = 1/2$.
The proof of Theorem~\ref{thm.variants} relies on the following scaling lemma, resembling Lemma~\ref{lem.scaling.radial}.

\begin{lemma}\label{lem.scaling.vertex} Let $\cC \subseteq \R^n$ be a polytope, let $f\colon \cC \to \R$ be convex and differentiable in an open set containing $\cC$. Then, for all $\xx\in\cC\setminus X^*$ and all $S\in \SS(\xx)$, it holds that
    \begin{equation}\label{eq.scaling.vertex}
        \max_{\aa\in S, \vv\in \cC} \ip{\nabla f(\xx)}{\aa-\vv} \geq \frac{f(\xx) - f^*}{\v(X^*,\xx)}.
    \end{equation}
    In particular, if $(\cC,f)$ satisfies a $(\mu,\theta)$-error bound relative to $\v$ for some $\mu > 0$ and $\theta\in [0,1/2]$, then for all $\xx\in \cC$
    \begin{equation}\label{eq.decrease.vertex}
        \max_{\aa\in S,\vv\in \cC} \ip{\nabla f(\xx)}{\aa-\vv} \geq \mu^\theta(f(\xx)-f^*)^{1-\theta}.
    \end{equation}
\end{lemma}
\begin{proof}{Proof.} 
Suppose $\xx\in\cC\setminus X^*$ and $S\in \SS(\xx)$.
Let $\xx^*\in X^*$ be such that $\v(\xx^*,\xx) = \v(X^*,\xx)$. The definition of $\v$ implies that there exist $\ww\in\cC, \uu \in \conv(S)$ such that $\xx^*-\xx = \gamma(\ww-\uu)$ for some $0 < \gamma \le \v(\xx^*,\xx) = \v(X^*,\xx)$. Thus,
    \begin{align*}
        \max_{\aa\in S, \vv\in \cC} \ip{\nabla f(\xx)}{\aa-\vv} \ge \ip{\nabla f(\xx)}{\uu-\ww} = \frac{\ip{\nabla f(\xx)}{\xx-\xx^*}}{\gamma}  \ge  \frac{f(\xx) - f^*}{\v(X^*,\xx)},
    \end{align*}
    where the last inequality follows from convexity of $f$ and $0<\gamma \le \v(X^*,\xx)$. Thus~\eqref{eq.scaling.vertex} follows.  To show~\eqref{eq.decrease.vertex}, suppose $\xx \in \cC$.  If $\xx\in X^*$ then~\eqref{eq.decrease.vertex}  trivially holds.  Otherwise, when $\xx \in \cC\setminus X^*$,    
observe that~\eqref{eq.scaling.vertex} and the $(\mu,\theta)$-error bound relative to $\v$ imply that
\[
        \max_{\aa\in S, \vv\in \cC} \ip{\nabla f(\xx)}{\aa-\vv}  \ge  \frac{f(\xx) - f^*}{\v(X^*,\xx)}                                =\frac{(f(\xx) - f^*)^{1-\theta}(f(\xx) - f^*)^\theta}{\v(X^*,\xx)}  \ge \mu^\theta(f(\xx) - f^*)^{1-\theta}.
\]

\end{proof}

\begin{proof}{Proof of Theorem~\ref{thm.variants}.}
 The choice of $\dd^{(t)}$ implies that for $t\in \N$
    \begin{equation}\label{eq.angle.1}
        \ip{\nabla f(\xx^{(t)})}{ \dd^{(t)}} \le  \frac{\ip{\nabla f(\xx^{(t)})}{\vv^{(t)}-\aa^{(t)}}}{2} < 0
    \end{equation}
    and
    \begin{equation}\label{eq.angle.2}
        \ip{\nabla f(\xx^{(t)})}{\dd^{(t)}} \le \ip{\nabla f(\xx^{(t)})}{\vv^{(t)}-\xx^{(t)}} 
        \le\ip{\nabla f(\xx^{(t)})}{\xx^{*}-\xx^{(t)}} 
        \le f^*-f(\xx).
    \end{equation}
    We next proceed by considering three possible cases:
    \begin{enumerate}
        \item $\eta_t < \eta_{t,\max}$. In this case,~\eqref{eq.eta.loose},~\eqref{eq.angle.1}, and~\eqref{eq.decrease.vertex} imply that
              \begin{align*}
                  f(\xx^{(t+1)}) & \leq f(\xx^{(t)}) - \frac{\langle \nabla f(\xx^{(t)}), \dd^{(t)}\rangle^2}{2L} \\&
                  \le f(\xx^{(t)}) - \frac{\ip{\nabla f(\xx^{(t)})}{\vv^{(t)}-\aa^{(t)}}^2}{8L} \notag            \\ &\le f(\xx^{(t)}) - \frac{\mu^{2\theta}(f(\xx^{(t)}) - f^*)^{2(1-\theta)}}{8L}.
              \end{align*}
        \item $\eta_t = \eta_{t,\max}\geq 1$.  In this case,~\eqref{eq.eta.tight} and~\eqref{eq.angle.2} imply that
\[
                  f(\xx^{(t+1)}) \le f(\xx^{(t)}) + \frac{\ip{\nabla f(\xx^{(t)})}{\dd^{(t)}}}{2} \le
                  f(\xx^{(t)})-\frac{f(\xx^{(t)})-f^*}{2}.
\]
           Thus, if either of Case 1 or Case 2 occurs, we get
              \begin{equation}\label{eq.decrease.fwa}
                  f(\xx^{(t+1)}) - f^*\leq (f(\xx^{(t)})-f^*)\left(1- \min\left\{ \frac{\mu^{2\theta}(f(\xx^{(t)}) - f^*)^{1-2\theta}}{8L},\frac{1}{2} \right\}\right).\end{equation}
        \item $\eta_t = \eta_{t,\max} < 1$. In this case, the choice of step-size $\eta_t$ implies that
\[
                  f(\xx^{(t+1)}) - f^*\leq f(\xx^{(t)}) - f^*.
\]
    \end{enumerate}
    Whenever Case 3 occurs we have $|S(\llambda^{(t+1)})| \le |S(\llambda^{(t)})|-1$ and whenever Case 1 or Case 2 occurs we have $|S(\llambda^{(t+1)})| \le |S(\llambda^{(t)})|+1$. Since $|S(\llambda^{(t)})| \geq 1$ for all $t\in\N$, it follows that for any consecutive sequence of iterations $\{t_0,t_0+1,\dots,t\}$  with $t > |S(\llambda^{(t_0)})|+t_0-1$ Case 1 or Case 2 must occur at least $\lceil
        \frac{t-|S(\llambda^{(t_0)}|-t_0+1}{2} \rceil$ times.  Furthermore, the sequence $f(\xx^{(t)}) - f^*$ is monotonically nonincreasing as this is guaranteed in any of the three cases.  The proof is completed as follows.
    When $\theta = 1/2$, Inequality~\eqref{eq.decrease.fwa} and the monotonicity of $f(\xx^{(t)}) - f^*$ imply that for  $t\in\N$ whenever Case 1 or Case 2 occurs, it holds that
    \[
        f(\xx^{(t+1)}) - f^*\leq (f(\xx^{(t)})-f^*)\left(1- \min\left\{
        \frac{\mu}{8L},\frac{1}{2}
        \right\}\right).
    \]
 Thus~\eqref{eq.linear.fwa} follows.
    When $\theta \in [0,1/2)$, Inequality~\eqref{eq.decrease.fwa} and the monotonicity of
    $f(\xx^{(t)}) - f^*$ imply that whenever Case 1 or Case 2 occurs and $t<t_0$, it holds that
    \[
        f(\xx^{(t+1)}) - f^*\leq \frac{1}{2}(f(\xx^{(t)}) - f^*).
    \]
    Therefore,~\eqref{eq.linearregime.fwa} holds for $t=0,1,\dots,t_0$.
    Finally, when $\theta \in [0,1/2)$, Inequality~\eqref{eq.decrease.fwa} and the monotonicity of
    $f(\xx^{(t)}) - f^*$ imply that
    whenever Case 1 or Case 2 occurs and $t\ge t_0$, it holds that that
    \[
        f(\xx^{(t+1)}) - f^*\leq(f(\xx^{(t)})-f^*)
        \left(1-\frac{\mu^{2\theta}}{8L}(f(\xx^{(t)}) - f^*)^{1-2\theta}\right).
    \]
 Thus~\eqref{eq.constantregime.fwa} holds for $t_0 < t \le 
 t_1 = |S(\llambda^{(t_0)})|+t_0-1$~and Lemma~\ref{lemma:borwein} implies~\eqref{eq.sublinearregime.fwa} for $t > t_1$.
\end{proof}

We next discuss common sufficient conditions for the the assumptions of Theorem~\ref{thm.variants} to hold.  First, as noted in Section~\ref{sec.prelims}, the pair $(\cC,f)$ has extended $L$-curvature for some $L > 0$ whenever $f$ is smooth with respect to a norm in $\cC$.
Proposition~\ref{prop:aff_err.vertex} below is an analogue of Proposition~\ref{prop:int_aff_err}.  It shows that $(\cC,f)$ satisfies an error bound relative to $\v$ provided $\cC$ is a polytope and $(\cC,f)$ satisfies a  Hölderian error bound with respect to some norm in $\cC$.

The proof of Proposition~\ref{prop:aff_err.vertex} relies on the following  definition and technical lemma.  

\begin{definition}[Minimal face]\label{def.minimal_face}
    Let $\cC\subseteq\R^n$ be a nonempty polytope.  For $\xx\in\cC$ let $F(\xx) \in\faces(\cC)$ denote the \emph{minimal face} of $\cC$ that contains $\xx$, that is, $F(\xx) \in \faces(\cC)$ satisfies $\xx\in F(\xx)$ and $ F(\xx)\subseteq G$ for all $G\in\faces(\cC)$ such that $\xx\in G$.
Similarly, for $\emptyset \ne X\subseteq\cC$ let $F(X)\in \faces(\cC)$ denote the \emph{minimal face} of $\cC$ that contains $X$.

\end{definition}

\begin{lemma}\label{lemma.vertex.facial}
Let $\R^n$ be endowed with a norm $\|\cdot\|$ and let $\cC \subseteq \R^n$ be a nonempty polytope.
 Then for all $\xx,\yy\in\cC$
\begin{equation}\label{eq.vertex.facial}
                 \|\yy-\xx\|\ge \v(\yy,\xx) \cdot \Phi(F(\yy),\cC).
\end{equation}

\end{lemma}
For ease of exposition, we defer the proof of Lemma~\ref{lemma.vertex.facial} to Section~\ref{sec.facial_geometry}.

\begin{proposition}[Vertex distance error bound]\label{prop:aff_err.vertex}
Let $\R^n$ be endowed with a norm $\|\cdot\|$, let $\cC \subseteq \R^n$ be a polytope, and let $f\colon\cC\to\R$ be convex.
    Suppose that $(\cC,f)$ satisfies a $(\mu, \theta)$-Hölderian error bound with respect to $\|\cdot\|$ for some $\mu>0$ and $\theta \in(0,1/2]$. Then
    $(\cC, f)$ satisfies a $(\tilde \mu,\theta)$ error bound relative to $\v$ for 
    \[
    \tilde \mu = \mu \cdot \Phi(F(X^*),\cC)^{1/\theta}.
    \]
\end{proposition}
\begin{proof}{Proof.}
This proof is similar to that of Proposition~\ref{prop:int_aff_err}. 
Lemma~\ref{lemma.vertex.facial} implies that 
for all $\xx\in \cC$
\[
\min_{\xx^*\in X^*}\|\xx^*-\xx\| \ge \v(X^*,\xx) \cdot \Phi(F(X^*),\cC).
\]
Since $f$ satisfies a $(\mu, \theta)$-Hölderian error bound over $\cC$ with respect to the norm $\|\cdot\|$, it follows that for all $\xx\in \cC$
\[
\left(\frac{f(\xx)-f^*}{\mu}\right)^\theta \geq \min_{\xx^*\in X^*} \|\xx^*-\xx\| \ge \v(X^*,\xx) \cdot  \Phi(F(X^*),\cC),
\]
and thus 
\[
\left(\frac{f(\xx)-f^*}{\tilde \mu}\right)^\theta \geq \v(X^*,\xx).
\]
\end{proof}

Note that in the seminal work of \cite{lacoste2015global}, the convergence rate depended on the global inner facial distance instead. The global facial distance generally depends on the dimension of the polytope $\cC$. This has already been demonstrated, for example, for the $\ell_1$-ball \cite{wirth2023approximate} and the hypercube \cite{pena2019polytope}. In Section~\ref{sec.lower_bound_fd}, we demonstrate that the same does not hold true for the local facial distance. Instead, we will show that the local facial distance depends only on the dimension of the optimal face $F(X^*)$. Our theory is the first result that gives this without additional assumptions such as strict complementarity as in, for example, \cite{garber2020revisiting}.

\subsection{In-face Frank-Wolfe algorithm}\label{sec.in_face}

\begin{algorithm}[t]
    \caption{In-face Frank-Wolfe algorithm (IFW)}\label{alg:in_face_fw}
\begin{algorithmic}
    \State {{\bf Input: }{$\xx^{(0)}\in\cC$}
    }
    \State { {\bf Output: }{$\xx^{(t)}\in\cC$  for $t\in\N$.}}
    \hrulealg
    \For{$t\in\N$}
    \State{compute $\dd^{(t)}_{\text{FW}}, \dd^{(t)}_{\text{inAFW}}, \; \dd^{(t)}_{\text{inBPFW}}$ as in~\eqref{eq.inface.directions} for $\xx = \xx^{(t)}$} 
    \State
    {select  $\dd^{(t)}\in \{\dd^{(t)}_{\text{FW}}, \dd^{(t)}_{\text{inAFW}},\dd^{(t)}_{\text{inBPFW}}\}$  that minimizes $\ip{\nabla f(\xx^{(t)})}{\dd^{(t)}}$}
    \State
    {compute $\eta_{t,\max}$ as in~\eqref{eq.inface.max.step-size}
    for $\xx = \xx^{(t)}$ and $\dd = \dd^{(t)}$}
    \State{$\xx^{(t+1)} \gets \xx^{(t)} + \eta_t \dd^{(t)}$ for some $\eta_t \in [0,\eta_{t,\max}]$}
    \EndFor
\end{algorithmic}
\end{algorithm}

We next describe another variant of the Frank-Wolfe algorithm based on \emph{in-face away, in-face pairwise}, and \emph{in-face blended pairwise} directions when a suitable \emph{in-face} oracle is available.  Variants of the Frank-Wolfe algorithm that incorporate in-face directions have been considered and analyzed by~\cite{freund2017extended} for low-rank matrix completion and by~\cite{garber2016linear} for minimization over simplex-like polytopes.  An  advantage of in-face variants is that they do not require an explicit vertex representation of the iterates.

Throughout this subsection we assume that $\cC \subseteq \R^n$ is a polytope  equipped with the following kind of \emph{in-face linear minimization oracle:} for $\xx\in \cC$ and $\gg \in \R^n$ there is an oracle that computes
\begin{equation}\label{eq.inface.oracle}
    \argmin_{\yy\in F(\xx)} \ip{\gg}{\yy},
\end{equation}
where $F(\xx)\in \faces(\cC)$ denotes the minimal face of $\cC$ that contains $\xx$ as detailed in Definition~\ref{def.minimal_face}.  
We also assume that $\cC$ is equipped with the following kind of \emph{maximum step-size oracle:} for $\xx\in\cC \subseteq \R^n$ and $\dd \in \cC-\cC$ there is an oracle that computes
\begin{equation}\label{eq.inface.max.step-size}
    \max\{\eta \ge 0 \mid \xx+ \eta \dd \in \cC\}.
\end{equation}
As it is discussed in~\cite{garber2016linear},  the above 
in-face linear minimization and maximum step-size
oracles can be constructed when a linear minimization oracle for $\cC$ is available and $\cC$ has a polyhedral description of the form
\[
    \cC = \{\xx\in\R^n\mid  A\xx = \bb, D\xx \ge \ee\},
\]
provided there is an additional available oracle for the mapping $\xx \mapsto D\xx$.  In particular, this is the case when $\cC$ has a standard form description
\[\cC = \{\xx\in\R^n \mid A\xx = \bb, \xx \geq \0\}.\]

The balls defined by the $\ell_1$- and $\ell_{\infty}$-norms have both in-face linear minimization and maximum step-size oracles.  In addition, as it is detailed in~\cite[Section 4]{garber2016linear}, the following are some examples of polytopes with a standard form description that have both in-face linear minimization and maximum step-size oracles: the standard simplex, the flow polytope, the perfect matchings polytope, and the marginal polytope.

For $\xx\in\cC$ define the \emph{in-face away} vertex $\aa \in S$ and \emph{in-face local minimizer} vertex $\zz\in S$, and \emph{global minimizer} vertex as follows
\[
    \aa  = \argmax_{\yy\in F(\xx)} \ip{\nabla f(\xx)}{\yy}, \qquad  \zz = \argmin_{\yy\in F(\xx)} \ip{\nabla f(\xx)}{\yy}, \qquad
    \vv  = \argmin_{\yy\in \cC} \ip{\nabla f(\xx)}{\yy}.
\]
The objects defined above in turn yield the following \emph{Frank-Wolfe}, \emph{in-face away}, and \emph{in-face blended pairwise directions} for any $\xx\in \cC$:
\begin{equation}\label{eq.inface.directions}
    \dd_{\text{FW}} = \vv-\xx,\qquad \dd_{\text{inAFW}} = \xx-\aa, \qquad \dd_{\text{inBPFW}} = \zz-\aa.
\end{equation}
We next establish affine-invariant convergence results for  Algorithm~\ref{alg:in_face_fw} analogous to Theorem~\ref{thm.vanilla} and Theorem~\ref{thm.variants}.  To do so, we rely on the following \emph{face distance}.

\begin{definition}[Face distance]\label{def.in_face_vertex_distance}
    Let $\cC\subseteq\R^n$ be a polytope. Define the \emph{face distance} $\f\colon\cC\times\cC \to [0,1]$ as follows.  For $\yy,\xx\in\cC$ let the face distance between $\yy$ and $\xx$ be
\[
        \f(\yy,\xx) = \min\{\gamma \geq 0 \mid \yy-\xx = \gamma (\vv-\uu)  \ \text{for some } \vv\in C \ \text{and} \ \uu\in F(\xx)\}.
\]
\end{definition}
Note that $\f(\yy, \xx)= 0$ if and only if $\yy = \xx$.
Observe that if $\cC\subseteq\R^n$ is a polytope, then for all $\yy,\xx\in \cC$, it holds that
$
    \f(\yy,\xx) \leq \v(\yy,\xx).
$ 
In particular, if $(\cC,f)$ satisfies the $(\mu,\theta)$-error bound relative to $\v$ then $(\cC,f)$ also satisfies the $(\mu,\theta)$-error bound relative to $\f$ and it would typically do so for a larger $\mu$.

Figure~\ref{fig:face} displays visualizations of the function $\f(\yy,\cdot)$ on $\cC$ for $\cC=[0,1]\times [0,1],\; \yy = (0,0)$ and also for $\cC=[0,2]\times [0,1],\; \yy = (0,0)$.
As in Figure~\ref{fig:radial} and Figure~\ref{fig:vertex}, Figure~\ref{fig:face}(b) is a stretched version of Figure~\ref{fig:face}(a) due of the affine invariance of $\f$.

\begin{figure}[ht]
\captionsetup[subfigure]{justification=centering}
 \hspace{-.6in}
\begin{tabular}{c c} 
\begin{subfigure}{.55\textwidth}
    \centering
        \includegraphics[width=.62\textwidth]{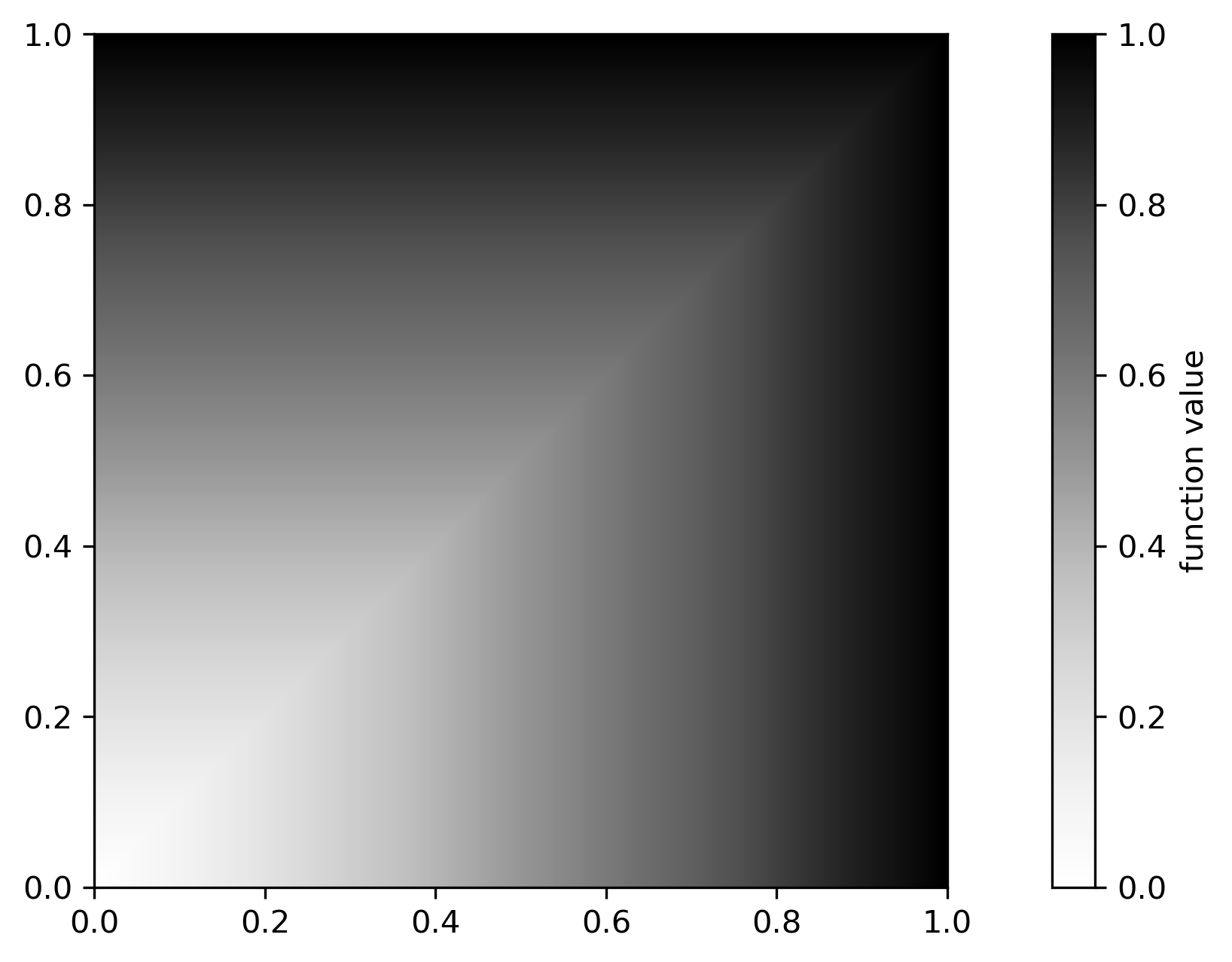}
        \caption{$\cC = [0,1]\times [0,1], \yy = (0,0)$.}
    \end{subfigure} & \hspace{-.5in}
    \begin{subfigure}{.575\textwidth}
    \centering
        \includegraphics[width=1\textwidth]{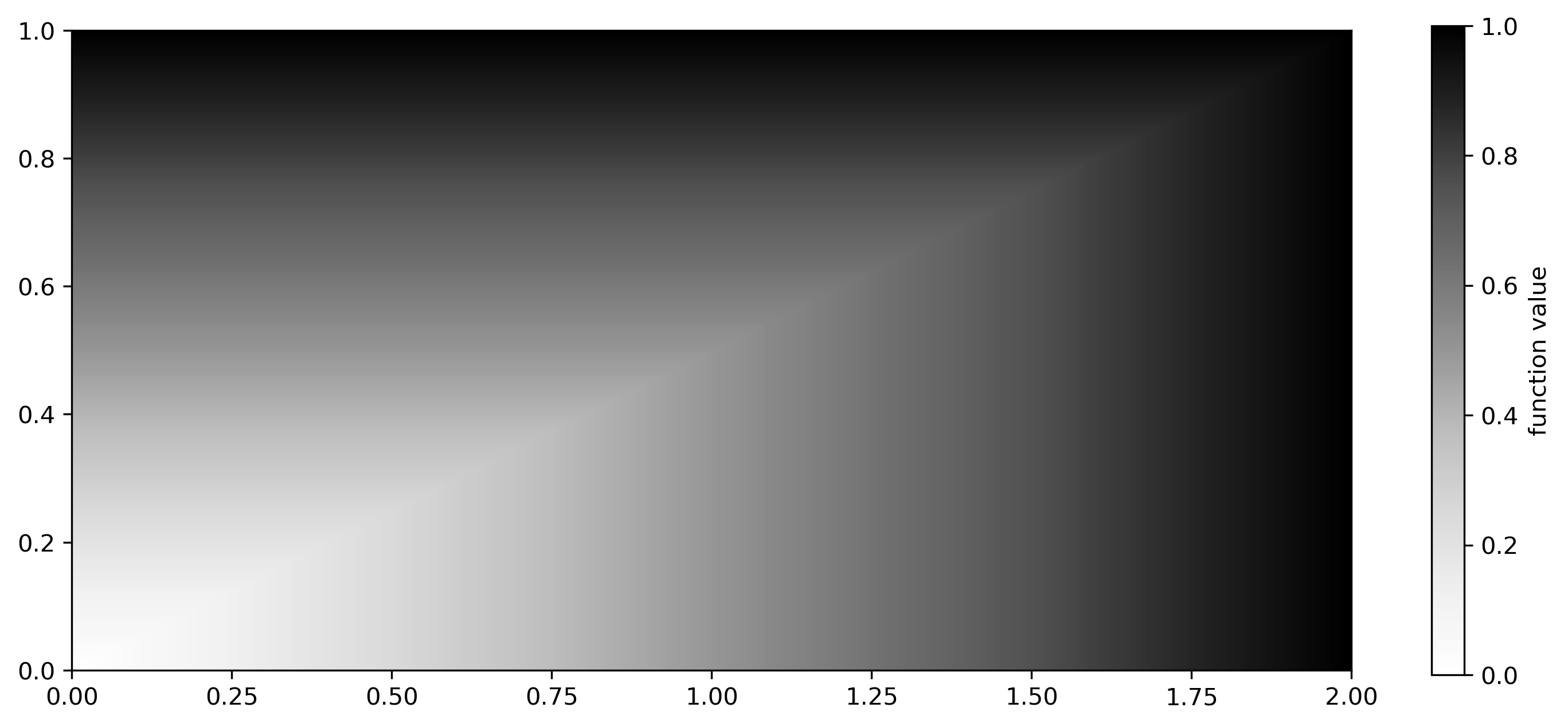}
        \caption{$\cC = [0,2]\times [0,1], \yy = (0,0)$.}
    \end{subfigure} 
\end{tabular}
\caption{
Visualization of the function
 $\f(\yy,\cdot)$ on $\cC$ for the face distance 
$\f$ of $\cC$.  The shade of gray indicates the function value. Points $\xx\in \cC$ with $\f(\yy,\xx) = 0$ are in white and points $\xx\in\cC$ with $\f(\yy,\xx)=1$ are in black.
}\label{fig:face}
\end{figure}

As we detail below, a suitable modification of
Theorem~\ref{thm.variants} yields the following convergence rate of Algorithm~\ref{alg:in_face_fw}.

\begin{theorem}[IFW]\label{thm.in_face_variants}
    Let $\cC\subseteq\R^n$ be a polytope  equipped with an in-face oracle~\eqref{eq.inface.oracle} and a maximum step-size oracle~\eqref{eq.inface.max.step-size}, let $f\colon \cC \to \R$ be convex and differentiable in an open set containing $\cC$, and suppose that $(\cC,f)$ has extended $L$-curvature for some $L>0$, and satisfies the $(\mu,\theta)$-error bound relative to $\f$ for some  $\mu > 0$ and $\theta\in[0,1/2]$.
Then the iterates of Algorithm~\ref{alg:in_face_fw} with line-search or short-step satisfy the following convergence rates.
    
\begin{itemize}
\item    
    When $\theta = 1/2$ the following linear convergence rate holds for all $t\in\N$:
    \begin{equation}\label{eq.in_face.linear.fwa}
        f(\xx^{(t)})-f^* \leq ( f(\xx^{(0)}) - f^* )\left(1-\min\left\{\frac{\mu}{8L}, \frac{1}{2}\right\}\right)^{\lceil t/\dim(\cC)\rceil}.
    \end{equation}
    \item
    When $\theta \in [0,1/2)$,     the following initial linear convergence rate holds for $t\in \{0,1,\dots,t_0\}$:
    \begin{equation}\label{eq.linearregime.fwa.inface}
        f(\xx^{(t)})-f^* \leq \frac{f(\xx^{(0)}) - f^*}{2^{\lceil t/\dim(\cC) \rceil}},
    \end{equation}
    where $t_0\in\N$ is the smallest $t$ such that
    $$\frac{\mu^{2\theta}(f(\xx^{(t)}) - f^*)^{1-2\theta}}{8L} \le \frac{1}{2}.$$
And for $t\in\N$ with $t > t_0$, the following sublinear convergence rate holds:
    \begin{equation}\label{eq.sublinearregime.fwa.inface}
         f(\xx^{(t)})-f^* \leq\left((f(\xx^{(t_0)}) -f^*)^{(2\theta-1)} + \frac{(1-2\theta)\mu^{2\theta}}{8L} \cdot \lceil \frac{
            t-t_0}{\dim(\cC)} \rceil\right)^{\frac{1}{2\theta-1}}
    \end{equation}
    \end{itemize}
\end{theorem}

The proof of Theorem~\ref{thm.in_face_variants} relies on the following analogue of Lemma~\ref{lem.scaling.vertex}.  We omit the proof of Lemma~\ref{lem.in_face.scaling.face} as it is a straightforward modification of the proof of Lemma~\ref{lem.scaling.vertex}.

\begin{lemma}\label{lem.in_face.scaling.face}  Let $\cC\subseteq\R^n$ be a polytope  and let $f\colon \cC \to \R$ be convex and differentiable in an open set containing $\cC$. Let $\xx\in \cC$ such that $\f(X^*,\xx) > 0$. Then,
    \begin{equation}\label{eq.in_face.scaling.face}
        \max_{\aa\in F(\xx), \vv \in \cC} \ip{\nabla f(\xx)}{\aa-\vv} \geq \frac{f(\xx) - f^*}{\f(X^*,\xx)}.
    \end{equation}
    In particular, if $(\cC,f)$ satisfies a $(\mu,\theta)$-error bound relative to $\f$ then for all $\xx\in \cC$
    \begin{equation}\label{eq.in_face.decrease.face}
        \max_{\aa\in F(\xx), \vv \in \cC} \ip{\nabla f(\xx)}{\aa-\vv} \geq \mu^\theta(f(\xx)-f^*))^{1-\theta}.
    \end{equation}
\end{lemma}
\begin{proof}{Proof of Theorem~\ref{thm.in_face_variants}.}
    This proof is similar that of Theorem~\ref{thm.variants} and thus we are deliberately succinct.
    The choice of $\dd^{(t)}$ implies that for all $t\in \N$
    \begin{equation}\label{eq.in_face.angle.1}
        \ip{\nabla f(\xx^{(t)})}{\dd^{(t)}} \le  \frac{\ip{\nabla f(\xx^{(t)})}{\vv^{(t)}-\aa^{(t)}}}{2} < 0
    \end{equation}
    and
    \begin{equation}\label{eq.in_face.angle.2}
        \ip{\nabla f(\xx^{(t)})}{\dd^{(t)}} \le \ip{\nabla f(\xx^{(t)})}{\vv^{(t)}-\xx^{(t)}} \le f^*-f(\xx).
    \end{equation}
    We proceed by considering three possible cases:

    \begin{enumerate}
        \item $\eta_t < \eta_{t,\max}$. In this case,~\eqref{eq.eta.loose},~\eqref{eq.in_face.angle.1}, and~\eqref{eq.in_face.decrease.face} imply that
\[
                  f(\xx^{(t+1)}) \leq f(\xx^{(t)}) - \frac{\mu^{2\theta}(f(\xx^{(t)}) - f^*)^{2(1-\theta)}}{8L}.
\]
        \item $\eta_t = \eta_{t,\max}\geq 1$.  In this case,~\eqref{eq.eta.tight} and~\eqref{eq.in_face.angle.2} imply that
\[
f(\xx^{(t+1)}) \leq f(\xx^{(t)})-\frac{f(\xx^{(t)})-f^*}{2}. 
\]
Thus if either of Case 1 or Case 2 occurs, then we get
              \[    f(\xx^{(t+1)}) - f^*\leq (f(\xx^{(t)})-f^*)\left(1- \min\left\{
                  \frac{\mu^{2\theta}(f(\xx^{(t)}) - f^*)^{1-2\theta}}{8L},\frac{1}{2}
                  \right\}\right). \]

        \item $\eta_t = \eta_{t,\max} < 1$. In this case the choice of step-size $\eta_t$ implies that
\[
                  f(\xx^{(t+1)}) - f^*\leq f(\xx^{(t)}) - f^*.
\]
    \end{enumerate}
    Whenever Case 3 occurs we have $\dim(F(\xx^{(t+1)}))\le \dim(F(\xx^{(t)}))-1$
    and whenever Case 1 or Case 2 occurs we have $\dim(F(\xx^{(t+1)}))\le \dim(\cC)$.
    Since $\dim(F(\xx^{(t)})) \geq 0$ for all $t\in\N$, Case 3 can occur at most $\dim(\cC)$ times per each occurrence of one of the first two cases.
 Thus for any consecutive sequence of iterations $\{t_0,t_0+1,\dots,t\}$ Case 1 or Case 2 must occur at least $\lceil
        \frac{t-t_0}{\dim(\cC)} \rceil$ times.  The rest of the proof is a straightforward modification of the last part of the proof of Theorem~\ref{thm.variants}.
\end{proof}

Proposition~\ref{prop:aff_err.face} below is an analogue of Proposition~\ref{prop:int_aff_err} and~\ref{prop:aff_err.vertex}.  It shows that $(\cC,f)$ satisfies an error bound relative to $\f$ provided $\cC$ is a polytope and $(\cC,f)$ satisfies a  Hölderian error bound with respect to some norm in $\cC$.  We omit the proof of Proposition~\ref{prop:aff_err.face} as it is a straightforward modification of the proof of Proposition~\ref{prop:aff_err.vertex} via the following analogue of Lemma~\ref{lemma.vertex.facial}.
\begin{lemma}\label{lemma.face.facial}
Let $\R^n$ be endowed with a norm $\|\cdot\|$ and let $\cC \subseteq \R^n$ be a nonempty polytope. Then for all $\xx,\yy\in\cC$
\begin{equation}\label{eq.face.facial}   
              \|\yy-\xx\|\ge \f(\yy,\xx) \cdot \bar \Phi(F(\yy),\cC).
\end{equation}
\end{lemma}
Again for ease of exposition, we defer the proof of Lemma~\ref{lemma.face.facial} to Section~\ref{sec.facial_geometry}.

\begin{proposition}[Face distance error bound]\label{prop:aff_err.face}
Let $\R^n$ be endowed with a norm $\|\cdot\|$, let $\cC \subseteq \R^n$ be a nonempty polytope, and let $f\colon\cC\to\R$ be convex.
    Suppose that $(\cC,f)$ satisfies a $(\mu, \theta)$-Hölderian error bound with respect to $\|\cdot\|$ for some $\mu>0$ and $\theta \in(0,1/2]$. Then
    $(\cC, f)$ satisfies a $(\tilde \mu,\theta)$ error bound relative to $\f$ for 
    \[
    \tilde \mu = \mu \cdot \bar \Phi(F(X^*),\cC)^{1/\theta}.
    \]
\end{proposition}

\subsubsection{In-face FW for polytopes in standard form}

So far, we discussed IFW for general polytopes. Notably, when a $(\mu,\theta)$-error bound holds with $\theta=1/2$, we derived the convergence guarantee $\cO(\exp(-\frac{t}{\dim(\cC)}))$. In the theorem below, when $\cC$ is a polytope in standard form, that is, 
\begin{equation}\label{eq.standard.form}
\cC = \{\xx\in\R^n \mid A\xx = \bb, \xx \geq \0\},
\end{equation}
where $A\in\R^{m\times n}$, $\rank(A) = m$, and $\bb\in\R^m$, we derive an alternate convergence rate of order $\cO(\exp(-\frac{t}{m}))$ which is sharper when $m < \dim(\cC)$.

\begin{theorem}[IFW for polytopes in standard form]\label{thm.in_face_variants_polytope}
Let $\cC\subseteq \R^n$ be a polytope of the form~\eqref{eq.standard.form}, 
 let $f\colon \cC \to \R$ be convex and differentiable in an open set containing $\cC$, and suppose that $(\cC,f)$ has extended $L$-curvature for some $L>0$,  and satisfies the $(\mu,\theta)$-error bound with $\theta \in [0,1/2]$ relative to $\f$ for some  $\mu > 0$ and $\theta\in[0,1/2]$. Suppose Algorithm~\ref{alg:in_face_fw} starts with $\xx^{(0)} \in \vertices(\cC)$. Then the iterates of Algorithm~\ref{alg:in_face_fw} with line-search or short-step satisfy the following convergece rates.
\begin{itemize}
\item    
    When $\theta = 1/2$ the following linear convergence rate holds for all $t\in\N$:
    \begin{align}\label{eq.in_face_polytope.linear.fwa}
        f(\xx^{(t)})-f^* \leq ( f(\xx^{(0)}) - f^* )\left(1-\min\left\{\frac{\mu}{8L}, \frac{1}{2}\right\}\right)^{\left\lceil t/m\right\rceil}.
    \end{align}
\item    When $\theta \in [0,1/2)$ the following initial linear convergence rate holds for $t\in \{0,1,\dots,t_0\}$:
    \begin{align}\label{eq.in_face_polytope.linearregime.fwa}
        f(\xx^{(t)})-f^* \leq \frac{f(\xx^{(0)}) -f^*}{2^{\lceil t/m \rceil}},
    \end{align}
    where $t_0\in\N$ is the smallest $t$ such that
    $$\frac{\mu^{2\theta}(f(\xx^{(t)}) - f^*)^{1-2\theta}}{8L} \le \frac{1}{2}.$$
For $t\in \N$ with $t_0 < t \le t_1:=\dim(F(\xx^{(t_0)}))+t_0$ the following  holds:
\begin{equation}\label{eq.constantregime.fwinface}
f(\xx^{(t)})-f^* \le f(\xx^{(t_0)}) -f^*.
\end{equation}
Finally, for $t\in \N$ with $t > t_1$,
the following sublinear convergence rate holds:
    \begin{equation}\label{eq.in_face_polytope.sublinearregime.fwa}
          f(\xx^{(t)})-f^* \leq\left((f(\xx^{(t_1)}) -f^*)^{(2\theta-1)} + \frac{(1-2\theta)\mu^{2\theta}}{8L} \cdot \left\lceil \frac{t-t_1}{m} \right\rceil\right)^{\frac{1}{2\theta-1}}.
    \end{equation}
    \end{itemize}
\end{theorem}
\begin{proof}{Proof.}
    The proof deviates from the proof of Theorem~\ref{thm.in_face_variants} only in the analysis of the number of drop steps. We focus on this aspect.

    Whenever Case 3 occurs we have $\dim(F(\xx^{(t+1)}))\le \dim(F(\xx^{(t)}))-1$.  On the other hand, whenever Case 1 or Case 2 occurs, the standard form description of $\cC$ implies that $\dim(F(\xx^{(t+1)}))\le \dim(F(\xx^{(t)}))+m$ because the reduction in the number of binding inequalities from $F(\xx^{(t)})$ to $F(\xx^{(t+1)})$ is at most $m$.
    Since $\dim(F(\xx^{(t)})) \geq 0$ for all $t\in\N$,
it follows that for any consecutive sequence of iterations $\{t_0,t_0+1,\dots,t\}$ with  $t >  t_1 = \dim(F(\xx^{(t_0)}))+t_0$ Case 1 or Case 2 must occur at least $\lceil
        \frac{t-t_1}{m} \rceil$ times.   The rest of the proof is analogous to the proof of Theorem~\ref{thm.in_face_variants}.
\end{proof}

\subsubsection{In-face pairwise Frank-Wolfe for simplex-like polytopes}
This section features an interesting connection with
\cite{garber2016linear} where Garber and Meshi introduced a decomposition-invariant pairwise Frank-Wolfe algorithm for \emph{simplex-like polytopes} whose convergence properties depends on the sparsity of the optimal solution set.

\begin{definition}[Simplex-like polytope]\label{def.slp}
    A polytope $\cC = \{\xx\in\R^n \mid A\xx = \bb, \xx\geq 0\}$ whose vertices are contained in $\{0, 1\}^n$, where $A\in\R^{m\times n}$ and $\bb\in\R^m$, is called a \emph{simplex-like polytope}.
\end{definition}
We present a FW variant that is similar to the decomposition-invariant pairwise Frank-Wolfe algorithm of Garber and Meshi~\cite{garber2016linear} in Algorithm~\ref{alg:in_face_fw.slp}.
Like Algorithm~\ref{alg:in_face_fw}, Algorithm~\ref{alg:in_face_fw.slp} relies on the in-face linear minimization oracle \eqref{eq.inface.oracle}.
Unlike Algorithm~\ref{alg:in_face_fw}, Algorithm~\ref{alg:in_face_fw.slp} uses the \emph{pairwise direction}:
\begin{equation}\label{eq.inface.pairwise}
    \dd_{\text{PW}}  = \vv-\aa,
\end{equation}
where $\aa = \argmax_{\yy\in F(\xx)} \ip{\nabla f(\xx)}{\yy}$ and $\vv = \argmin_{\yy\in \cC} \ip{\nabla f(\xx)}{\yy}$.  Algorithm~\ref{alg:in_face_fw.slp} also uses the following target short-step rule
\[
 - \frac{\langle \nabla f(\xx), \dd_\text{PW} \rangle}{L} = 
\argmin_{\gamma \ge 0}\left\{ 
f(\xx) + \gamma\langle \nabla f(\xx), \dd_\text{PW} \rangle+\frac{L\gamma^2}{2}
\right\}.
\]

\begin{algorithm}[t]
    \caption{Frank-Wolfe algorithm with in-face pairwise directions for simplex-like-polytopes (FWIPW)}\label{alg:in_face_fw.slp}
\begin{algorithmic}
    \State{\bf Input: }{$\xx^{(0)}\in\vertices(\cC)$
    }
    \State{\bf Output: }{$\xx^{(t)}\in\cC$  for $t\in\N$.}
    \hrulealg
\State    $\eta_{-1}\gets 1$
    \For{$t\in\N$}
\State    {compute $\dd^{(t)}_{\text{PW}}$ as in~\eqref{eq.inface.pairwise} for $\xx = \xx^{(t)}$} 
\State
    {$\gamma_t \gets - \frac{\langle \nabla f(\xx^{(t)}), \dd^{(t)}_\text{PW} \rangle}{L}$}
    \State
    {$\eta_t \gets \max \{2^{-k_t} \mid k_t\in\N, 2^{-k_t} \le \min\{\gamma_t, \eta_{t-1}\}\}$} \State
    {$\xx^{(t+1)} \gets \xx^{(t)} + \eta_t \dd^{(t)}_{\text{PW}}$}
    \EndFor
\end{algorithmic}    %
\end{algorithm}

We next establish affine-invariant convergence results for Algorithm~\ref{alg:in_face_fw.slp} analogous to but stronger than those in Theorems~\ref{thm.variants},~\ref{thm.in_face_variants}, and~\ref{thm.in_face_variants_polytope}.

\begin{theorem}[FWIPW]\label{thm.in_face_variants.slp}
    Let $\cC\subseteq\R^n$ be a simplex-like polytope, let $f\colon \cC \to \R$ be convex and differentiable in an open set containing $\cC$, and suppose that $(\cC,f)$ has extended $L$-curvature for some $L > 0$, and satisfies the $(\mu,\theta)$-error bound relative to $\f$ for some $\mu > 0$ and $\theta\in [0,1/2]$.  Then the iterates of Algorithm~\ref{alg:in_face_fw.slp} satisfy the following initial linear convergence rate for $t\in \{0,1,\dots,t_0\}$:
    \begin{equation}\label{eq.linearregime.in_face_fw.slp}
        f(\xx^{(t)})-f^* \leq\frac{f(\xx^{(0)}) - f^*}{2^t},
    \end{equation}
    where $t_0\in\N$ is the smallest $t$ such that
    $$-\langle \nabla f(\xx^{(t)}), \dd^{(t)}_\text{PW}\rangle < L.$$
Then  the following convergence rates hold.    
\begin{itemize}
\item
    When $\theta = 1/2$ the following linear convergence rate holds for all  $t \in \N_{\ge t_0}$:
\[
        f(\xx^{(t)}) - f^* \leq ( f(\xx^{(t_0)}) - f^* ) \left(1-\frac{\mu}{4L}\right)^{t-t_0}.
\]
\item    When $\theta \in [0,1/2)$ the following sublinear convergence rate holds for all  $t \in \N_{\ge t_0}$:
\[
        f(\xx^{(t)}) - f^* \leq \left( (f(\xx^{(t_0)})-f^*)^{(2\theta-1)} + \frac{(1-2\theta)\mu^{2\theta}(t-t_0)}{4L} \right)^{\frac{1}{2\theta-1}}.
\]
    \end{itemize}
\end{theorem}
\begin{proof}{Proof.}
    We first prove that the iterates are feasible, that is, $\xx^{(t)}\in\cC$ for all $t\in\N$.  We do so via the following two claims that are essentially identical to Observation 2 and Lemma 1 in~\cite{garber2016linear}.\\    
      
    \noindent
    \textbf{Claim 1:} Suppose that for some iteration $t\in\N$ of Algorithm~\ref{alg:in_face_fw.slp}, it holds that $\xx^{(t)}\in\cC$ and 
 $x_i^{(t)} \geq \eta_t$ whenever $x_i^{(t)}>0$.  Then $\xx^{(t+1)}\in\cC$.\\
    \noindent
    \textit{Proof of Claim 1:} Since
 $\aa^{(t)} = \argmin_{\yy \in F(\xx^{(t)})}\ip{\nabla f(\xx^{(t)})}{\yy}$, it follows that $a_i^{(t)} = 0$ whenever $x_i^{(t)} = 0$. Therefore since $\aa^{(t)}\in \{0,1\}^n$, our assumption on the positive entries of $\xx^{(t)}$ implies that $\xx^{(t)} -\eta_t\aa^{(t)} \ge 0.$  Since $\vv^{(t)}\in\cC$, in particular $\vv^{(t)} \ge \0$ and thus $\xx^{(t+1)} = \xx^{(t)} -\eta_t\aa^{(t)}+\eta_t \vv^{(t)} \geq \0$. Furthermore, since $\xx^{(t)}, \aa^{(t)},\vv^{(t)}\in\cC$, in particular $A\xx^{(t)} = A \aa^{(t)}= A\vv^{(t)}=\bb$ and hence  $A\xx^{(t+1)} = \bb$ as well. Therefore, $\xx^{(t+1)}\in\cC$.
    \\
    
    \noindent
    \textbf{Claim 2:}
    The iterates of Algorithm~\ref{alg:in_face_fw.slp} are feasible, that is, $\xx^{(t)}\in\cC$ for all $t\in\N$.\\
    \noindent
    \textit{Proof of Claim 2:}
    We prove by induction that for all $t\in\N$, there exists $\aalpha^{(t)}\in\N^n$ such that $\xx^{(t)}= \eta_{t} \aalpha^{(t)}$. The feasibility of the iterates then follows from Claim 1.
    For the base case $t=0$ observe that since $\xx^{(0)}$ is a vertex of $\cC$, it holds that $x_i^{(0)}\in\{0,1\}$ for $i=1,\dots,n$.  Since $\eta_0 = 2^{-k_0}$ for some $k_0\in\N$, it follows that $\xx^{(0)}= \eta_{0} \aalpha^{(0)}$ for $\aalpha^{(0)} = 2^{k_0}   \xx^{(0)}\in\N^n$. Suppose that the induction hypothesis holds for some $t\in\N$. Since subtracting $\eta_t\aa^{(t)}$ from $\xx^{(t)}$ can only decrease positive entries in $\xx^{(t)}$, see the proof of Claim 1, and since both $\aa^{(t)}$ and $\vv^{(t)}$ are vertices of $\cC$, the $i$-th entry of $\xx^{(t+1)}$ is given by
    \begin{align*}x_i^{(t+1)} = \eta_{t} \cdot
        \begin{cases}
            \alpha_i^{(t)},     & \text{if} \ \alpha_i(t)\geq 1 \ \text{and} \ a_i^{(t)}=v_i^{(t)}\in\{0,1\}            \\
            \alpha_i^{(t)} - 1, & \text{if} \ \alpha_i(t)\geq 1 \ \text{and} \ a_i^{(t)}=1 \ \text{and} \ v_i^{(t)} = 0 \\
            \alpha_i^{(t)} + 1, & \text{if} \ a_i^{(t)}= 0 \ \text{and} \ v_i^{(t)} = 1.
        \end{cases}
    \end{align*}
 Thus $\xx^{(t+1)} = \eta_t \tilde \aalpha^{(t)}$ for some $\tilde\aalpha^{(t)} \in \N^n$.
    Since $\eta_{t+1} = 2^{-k_{t+1}} \leq 2^{-k_t}=\eta_t$ for some $k_{t+1}\in\N$, it follows that $\frac{\eta_{t}}{\eta_{t+1}}\in\N_{\geq 1}$. Thus, $\aalpha^{(t+1)} = \frac{\eta_{t}}{\eta_{t+1}} \tilde\aalpha^{(t)}\in\N^n$ is such that $\xx^{(t+1)}= \eta_{t+1} \aalpha^{(t+1)}$, proving Claim 2.\\

    It remains to prove the convergence rates. To do so we rely on the following result.\\

    \noindent
    \textbf{Claim 3:} For all $t \in \N_{\ge t_0}$, it holds that
    \begin{equation}\label{eq.eta.slp}
        - \frac{\langle \nabla f(\xx^{(t)}), \dd^{(t)}_\text{PW} \rangle}{L} \geq \eta_t \geq \frac{\mu^\theta (f(\xx^{(t)}) - f^*)^{1-\theta}}{2L},
    \end{equation}
    which also implies that the  suboptimality gap is nonincreasing, that is, $f(\xx^{(t+1)}) \leq f(\xx^{(t)})$ for all $t \in \N_{\ge t_0}$.\\
    \noindent
    \textit{Proof of Claim 3:}
    The upper bound in~\eqref{eq.eta.slp} readily follows from the choice of $\eta_t$:
\[
        \eta_{t}\leq \gamma_{t} = - \frac{\langle \nabla f(\xx^{(t)}), \dd^{(t)}_\text{PW} \rangle}{L}.
\]
Furthermore, since $(\cC,f)$ has extended $L$-curvature it also follows that
\begin{align*}
f(\xx^{(t+1)}) &= f(\xx^{(t)}+ \eta_t \dd^{(t)}_\text{PW}) \\
&\le 
f(\xx^{(t)}) + \eta_t \ip{\nabla f(\xx^{(t)})}{\dd^{(t)}_\text{PW}} + \frac{L \eta_t^2}{2}\\
&\le f(\xx^{(t)}) - \frac{L\eta_t^2}{2} \\
& \le f(\xx^{(t)}).
\end{align*}   
Thus the suboptimality gap is nonincreasing.    
    
    We prove the lower bound in~\eqref{eq.eta.slp} by induction. The choice of $t_0$ and~\eqref{eq.in_face.decrease.face} imply that $$\eta_{t_0} \ge \frac{\gamma_{t_0}}{2} = - \frac{\langle \nabla f(\xx^{({t_0})}), \dd^{({t_0})}_\text{PW} \rangle}{2L} \geq \frac{\mu^\theta (f(\xx^{({t_0})}) - f^*)^{1-\theta}}{2L}.$$ 
    Suppose that the lower bound holds for $t\in\N_{\ge t_0}$. In case $\gamma_{t+1}\leq \eta_{t}$, Claim 3 follows from~\eqref{eq.in_face.decrease.face}:
\[
        \eta_{t+1} \geq \frac{\gamma_{t+1}}{2} = - \frac{\langle \nabla f(\xx^{(t+1)}), \dd^{(t+1)}_\text{PW} \rangle}{2L} \geq \frac{\mu^\theta (f(\xx^{(t+1)}) - f^*)^{1-\theta}}{2L}.
\]
    In case $\gamma_{t+1}\geq \eta_{t}$, Claim 3 follows from the induction hypothesis and  the nonincreasing suboptimality gap:
\[
        \eta_{t+1} = \eta_{t} \geq \frac{\mu^\theta (f(\xx^{(t)}) - f^*)^{1-\theta}}{2L} \geq \frac{\mu^\theta (f(\xx^{(t+1)}) - f^*)^{1-\theta}}{2L}.
\]

    Finally, we derive the convergence guarantee. For $t\in \{0,\dots,t_0\}$ we have $\eta_t = 1 \le - \langle \nabla f(\xx^{(t)}), \dd^{(t)}_{\text{PW}} \rangle$.  Thus the extended $L$-curvature of $(\cC,f)$ implies that
    \[
    f(\xx^{(t+1)}) \le  f(\xx^{(t)}) + \frac{\langle \nabla f(\xx^{(t)}), \dd^{(t)}_{\text{PW}} \rangle}{2} \le f(\xx^{(t)}) - \frac{f(\xx^{(t)})-f^*}{2}.
    \]
    Thus~\eqref{eq.linearregime.in_face_fw.slp} follows.

On the other hand, Claim 3, $L$-curvature, and~\eqref{eq.in_face.decrease.face} imply that for all $t\in\N_{\ge t_0}$ it holds that
    \begin{align*}
        f(\xx^{(t+1)}) & \leq f(\xx^{(t)}) +\eta_t \langle \nabla f(\xx^{(t)}), \dd^{(t)}_{\text{PW}} \rangle + \frac{L\eta_t^2}{2}             \\
                       & \leq f(\xx^{(t)}) +\eta_t \left( \langle \nabla f(\xx^{(t)}), \dd^{(t)}_{\text{PW}} \rangle + \frac{L\eta_t}{2}\right) \\
                       & \leq f(\xx^{(t)}) + \eta_t \frac{\langle \nabla f(\xx^{(t)}), \dd^{(t)}_{\text{PW}} \rangle}{2}                        \\
                       & \leq f(\xx^{(t)}) - \frac{\mu^{2\theta}(f(\xx^{(t)}) - f^*)^{2(1-\theta)}}{4L}.
    \end{align*}
    The theorem follows from repeated application of the above inequality when $\theta=1/2$ and from Lemma~\ref{lemma:borwein} when $\theta\in [0, 1/2)$.
\end{proof}

We conclude this section with the following result analogous to  Proposition~\ref{prop:int_aff_err},  Proposition~\ref{prop:aff_err.vertex}, and Proposition~\ref{prop:aff_err.face} but specialized to the Euclidean norm.
We rely on the following notation and terminology from~\cite{garber2016linear}.  Suppose $\cC\subseteq\R^n$ is a simplex-like polytope and $\emptyset \ne X \subseteq \cC$.   Let $\text{\rm card}(X)$  denote the cardinality of $X$, that is, the number of positive entries of the most dense point in $X$.

\begin{proposition}\label{prop:aff_err.face.simplex}
Let $\R^n$ be endowed with the Euclidean norm $\|\cdot\|_2$, let $\cC \subseteq \R^n$ be a polytope, and let $f\colon\cC\to\R$ be convex.
    Suppose that $(\cC,f)$ satisfies a $(\mu, \theta)$-Hölderian error bound with respect to $\|\cdot\|_2$ for some $\mu>0$ and $\theta \in(0,1/2]$. Then
    $(\cC, f)$ satisfies a $(\tilde \mu,\theta)$ error bound relative to $\f$ for 
    \[
    \tilde \mu = \max\left\{\frac{\mu}{\text{\rm card}(X^*)^{1/(2\theta)}},    \frac{\mu}{(n-\text{\rm card}(X^*))^{1/(2\theta)}} \right\}.
    \]
\end{proposition}

The bound in Proposition~\ref{prop:aff_err.face.simplex} is in the same spirit as the bound stated as Lemma~2 in \cite{garber2016linear}.  A difference is that the bound in Proposition~\ref{prop:aff_err.face.simplex} is sharper when the most dense point in $X^*$  has more than $n/2$ non-zero components.  We defer the proof of 
Proposition~\ref{prop:aff_err.face.simplex} to Section~\ref{sec.facial_geometry}.

\section{Facial geometry}\label{sec.facial_geometry}

In this section, we develop some geometric properties of the inner and outer facial distances.  More precisely, we give proofs of Lemma~\ref{lemma.vertex.facial} and Lemma~\ref{lemma.face.facial} and give generic lower bounds on the inner and outer facial distances. In particular, we give a proof of Proposition~\ref{prop:aff_err.face.simplex} as an immediate consequence of more general bounds.

\subsection{Proofs of Lemma~\ref{lemma.vertex.facial} and Lemma~\ref{lemma.face.facial}}\label{sec.geometric_properties}

\begin{proof}{Proof of Lemma~\ref{lemma.vertex.facial}.}
Suppose that $\xx,\yy \in \cC$ with $\yy\ne\xx$ as otherwise \eqref{eq.vertex.facial} trivially holds.  Consider the following modified version of $\v(\yy,\xx)$:
              \begin{equation}\label{eq.tilde.g}
                  \tilde \v(\yy,\xx) := \max_{S\in\SS(\xx)} \min\{\gamma \geq 0 \mid \;  \yy-\xx = \gamma (\vv-\uu) \ \text{for some} \ \vv\in F(\yy)
                                                                                        \text{and} \ \uu\in\conv(S)\}.
              \end{equation}
              Suppose that $S,\vv,\uu$ attain the value $\gamma = \tilde \v(\yy,\xx)$ in~\eqref{eq.tilde.g} and $S$ is a minimal set with respect to inclusion.

              We claim that the construction of $\tilde \v(\yy,\xx)$ and the choice of $\uu,\vv$ imply that $\uu \in \conv(\vertices(\cC)\setminus F(\vv))$. To show this claim it suffices to show that  $S\cap \vertices(F(\vv)) = \emptyset$.  To do so, we proceed by contradiction. Observe that the minimality of $S$ implies that $\uu = \sum_{\ss\in S}\lambda_{\ss} \ss$ for some $\lambda_{\ss}>0, \ss\in S$ with $\sum_{\ss\in S}\lambda_{\ss}=1.$  If $S\cap \vertices(F(\vv)) \ne \emptyset$ then for some $\ss\in S\cap\vertices(F(\vv))$ and $\epsilon > 0$ sufficiently small we have both $\uu':=(1+\epsilon) \uu - \epsilon \ss\in \conv(S)$ because all $\lambda_{\ss} > 0$ and $\vv':=(1+\epsilon) \vv - \epsilon \ss\in F(\vv)$ because $\vv\in\relint(F(\vv))$.  Hence we get
              \[
                  \yy-\xx = \frac{\gamma}{1+\epsilon}(\vv'-\uu')
              \]
              with $\gamma/(1+\epsilon) < \gamma$. This contradicts the choice of $S,\vv,\uu$.  Thus, $S\cap \vertices(F(\vv))=\emptyset$.
              To finish, observe that $\gamma = \tilde \v(\yy,\xx) \ge \v(\yy,\xx)$ and thus
\begin{align*}
\|\yy-\xx\| &= \gamma \|\vv-\uu\| \\
&\ge  \v(\yy,\xx) \cdot \dist(F(\vv),\conv(\vertices(\cC)\setminus F(\vv)))\\
&\ge\v(\yy,\xx) \cdot \Phi(F(\yy),\cC).
\end{align*}              
              The second to last step holds because $\emptyset\ne F(\vv) \in \faces(F(\yy))$ and $\uu\in\conv(\vertices(\cC)\setminus F(\vv))$.
              Since this holds for any $\xx,\yy \in \cC$ with $\yy\ne \xx$, the bound~\eqref{eq.vertex.facial} follows.
\end{proof}

\begin{proof}{Proof of Lemma~\ref{lemma.face.facial}.}
This proof is similar, but simpler, to that of Lemma~\ref{lemma.vertex.facial}.
              Suppose that $\xx,\yy \in \cC$ with $\yy\ne \xx$ 
as otherwise~\eqref{eq.face.facial} trivially holds.  Consider the following modified version of $\f(\yy,\xx)$:
              \begin{equation}\label{eq.tilde.f}
                  \bar \f(\yy,\xx) := \min\{\gamma \geq 0 \mid \;  \yy-\xx = \gamma (\vv-\uu) \ \text{for some} \ \vv\in F(\yy)   \ \text{and} \ \uu\in F(\xx)\}.
              \end{equation}              
              and suppose $\vv,\uu$ attain the value $\gamma = \bar \f(\yy,\xx)$ 
              in~\eqref{eq.tilde.f}, that is, $\vv\in F(\yy),\uu\in F(\xx)$ and
$\yy-\xx = \gamma (\vv-\uu).$
              The construction of $\bar \f(\yy,\xx)$ and the choice of $\gamma,\uu,\vv$ imply that $F(\uu) \cap F(\vv) = \emptyset$. To finish, observe that $\gamma =\bar \f(\yy,\xx)\ge  \f(\yy,\xx)$ and thus
\begin{align*}
\|\yy-\xx\| & = \gamma\|\vv-\uu\| \\
&\ge \f(\yy,\xx) \cdot \dist(F(\vv), F(\uu))\\
&\ge \f(\yy,\xx) \cdot \bar\Phi(F(\yy),\cC).
              \end{align*}
              The last step holds because both $F(\vv) \in \faces(F(\yy))$ and $F(\uu)\in \faces(F(\xx)) \subseteq \faces(\cC)$ are evidently nonempty, and $F(\uu) \cap F(\vv)=\emptyset$.
              Since this holds for any $\xx,\yy \in \cC$ with $\yy\ne \xx$, the bound~\eqref{eq.face.facial} follows. 
\end{proof}

\subsection{Lower bound on the inner and outer facial distances}\label{sec.lower_bound_fd}

We next provide lower bounds on the inner and outer facial distances of a polytope in terms of a polyhedral description of the polytope.  The bounds are stated in Corollary~\ref{cor.bound.facialdist} and Corollary~\ref{cor.bound.facialdist.std} below.  Throughout this section, we assume that $\cC\subseteq \R^n$ is a polytope with the following polyhedral description
\begin{equation}\label{eq.polyhedral}
    \cC = \{\xx\in\R^n\mid  A\xx = \bb, D\xx \ge \ee\},
\end{equation}
where $A\in\R^{m\times n}$, $\bb\in\R^m$, $D\in \R^{k\times n}$, and $\ee\in\R^k$. 
To ease notation, throughout the remaining of this section we will write $[k]$ as shorthand for the set $\{1,2,\dots,k\}$.

We shall assume that the description~\eqref{eq.polyhedral} satisfies the following: for each $i\in [k]$, there exists an $\xx\in \cC$ such that $\ip{\dd_i}{\xx} > e_i$. This can be assumed without loss of generality as otherwise the inequality $\ip{\dd_i}{\xx} \ge e_i$ in~\eqref{eq.polyhedral} could be replaced with $\ip{\dd_i}{\xx} = e_i$.
We also assume without loss of generality that $\min_{\xx\in \cC}\ip{\dd_i}{\xx} = e_i$, as otherwise the inequality $\ip{\dd_i}{\xx} \ge e_i$ is redundant and can be dropped.

For $i\in [k]$, let $\dd_i\in \R^n$ denote the $i$-th row of $D$ so that for all $\xx\in \R^n$
\[
    (D\xx)_i = \ip{\dd_i}{\xx}.
\]
Similarly, for $I\subseteq[k]$, let $D_I$ denote the $|I|\times n$ submatrix of $D$ obtained by selecting the rows of $D$ indexed by $I$.  Likewise,
for $I\subseteq[k]$, let $\ee_I$ denote the $|I|$ subvector of $\ee$ obtained by selecting the entries of $\ee$ indexed by $I$.
 
  We also rely on the following notation.

\begin{definition}\label{def.i_f_c_f}
    Suppose that $\cC\subseteq\R^n$ is a polytope with a polyhedral description~\eqref{eq.polyhedral}.  For each $i\in [k]$, let
    \[
        \sigma_i:=\min\{\ip{\dd_i}{\vv}-e_i \mid \vv\in \vertices(\cC) \ \text{and} \ \ip{\dd_i}{\vv}>e_i\}.
    \]
    Suppose that $F \in \faces(\cC)$ is such that $F \subsetneq \cC$.
    Let
\[
I_F:= \{i\in [k] \mid \ip{\dd_i}{\xx} = e_i \ \text{for all} \ \xx\in F\}.
\]
Observe that $I_F\ne \emptyset$ because $F \in \faces(\cC)$ and $F\ne \cC$.    Let $D_{I_F}$ denote the $|I_F|\times n$ matrix whose rows are the rows from $D$ indexed by $I_F$ and let 
    \[\mathbb{I}(F):= \{I\subseteq I_F \mid D_I  \ \text{is full row rank and } I \text{ is maximal}\}.\]
    In other words, $\mathbb{I}(F)$ is the collection of all maximal sets of  rows  from $D_{I_F}$ that are linearly independent.
\end{definition}

\begin{lemma}[Minimal face for polytopes of the form \eqref{eq.polyhedral}]\label{lemma.face}
    Let  $\cC\subseteq\R^n$ be a polytope of the form \eqref{eq.polyhedral} and let $\emptyset\neq F\subsetneq\cC$ be a face of $\cC$. Then for all $I\in \mathbb{I}(F)$ it holds that
\[
        F =  \cC \cap \{\xx \mid D_I \xx = \ee_I\} = \cC \cap \{\xx \mid D_{I_F} \xx = \ee_{I_F}\}.
\]
\end{lemma}
\begin{proof}{Proof.}
The construction of $I_F$ and $\mathbb{I}(F)$ immediately implies that 
$
        F \subseteq  \cC \cap \{\xx \mid D_I \xx = \ee_I\} = \cC \cap \{\xx \mid D_{I_F} \xx = \ee_{I_F}\}$
        for all $I\in \mathbb{I}(F)$.
For the reverse inclusion, again the construction of $I_F$ implies that if 
$\xx\in \cC \setminus F$ then $\ip{\dd_i}{\xx} > e_i$ for some $i\in I_F$.  Therefore
$
\cC\setminus F \subseteq \cC \setminus \{\xx \mid D_{I_F} \xx = \ee_{I_F}\}
$ and consequently $\cC \cap \{\xx \mid D_{I_F} \xx = \ee_{I_F}\} \subseteq \cC\cap F = F$.
\end{proof}

The following result provides the key components to bound the inner and outer facial distances that we subsequently state in Corollary~\ref{cor.bound.facialdist} and Corollary~\ref{cor.bound.facialdist.std}.

Recall that if $\R^n$ is endowed with a norm $\|\cdot\|$, then $\|\cdot\|^*$ denotes the dual norm defined via~\eqref{eq.dual.norm}.

\begin{theorem}[Lower bound on facial distance]\label{thm:difd}
    Let $\R^n$ be endowed with a norm $\|\cdot\|$ and suppose that $\cC\subseteq\R^n$ is a polytope with a description of the form~\eqref{eq.polyhedral}.
    \begin{enumerate}
        \item[(a)] If $F \in \faces(\cC)$ and $F \subsetneq \cC$ then
              \begin{equation}\label{eq.lower.bound.facial}
                  \dist(F,\conv(\vertices(\cC)\setminus F)) \geq
                  \max_{I\in \mathbb{I}(F)}
                  \frac{1}{\|\sum_{i\in I} \dd_i/\sigma_{i}\|^*}.
              \end{equation}
        \item[(b)] If $F,G \in \faces(\cC)$ and $F\cap G = \emptyset$ then
              \begin{equation}\label{eq.lower.bound.outer.facial}
                  \dist(F,G) \geq \max\left\{
                  \max_{I\in \mathbb{I}(F)}
                  \frac{1}{\|\sum_{i\in I\cap I_G^c} \dd_i/\sigma_{i}\|^*},
                  \max_{I\in \mathbb{I}(G)}
                  \frac{1}{\|\sum_{i\in I\cap I_F^c} \dd_i/\sigma_{i}\|^*}\right\}.
              \end{equation}
            \end{enumerate}
\end{theorem}
\begin{proof}{Proof.}
    \begin{enumerate}
        \item[(a)]
              Let $G:=\conv(\vertices(\cC)\setminus F)$ and suppose $I\in \mathbb{I}(F)$.  Lemma~\ref{lemma.face} implies that for any vertex $\uu\in\vertices(G)$ there exists $i\in I$ such that  $\ip{\dd_{i}}{\uu} - e_i\ge \sigma_{i}$ or equivalently $(\ip{\dd_{i}}{\uu} - e_i)/\sigma_{i}\ge 1$. 
Furthermore, $\ip{\dd_{i}}{\uu} - e_i\ge 0$ for all $i\in I$ and $\uu\in\vertices(G)$ since $ \vertices(G)\subseteq \cC$.  Hence for each $\uu\in\vertices(G)$ it holds that
\begin{equation}\label{eq.gap}
              \sum_{i\in I} (\ip{\dd_i}{\uu} - e_i)/\sigma_{i} \ge 1.
\end{equation}
              Now suppose $\xx\in F$ and $\yy\in G$. Since  $\xx\in F$, it follows that  $\ip{\dd_i}{\xx} = e_i$ for all $i\in I$ and so
              \[
              \sum_{i\in I} \ip{\dd_i/\sigma_{i}}{\yy - \xx} = \sum_{i\in I} (\ip{\dd_i}{\yy} - e_i)/\sigma_{i}.
              \]
Since $\yy$ is a convex combination of vertices of $G$, inequality~\eqref{eq.gap} further implies that  
\[\sum_{i\in I} \ip{\dd_i/\sigma_{i}}{\yy - \xx} = \sum_{i\in I} (\ip{\dd_i}{\yy} - e_i)/\sigma_{i} \ge 1.\]
           Thus, by Cauchy-Schwarz inequality,
\[
                  \|\yy-\xx\| \geq \frac{1}{\|\sum_{i\in I} \dd_i/\sigma_{i}\|^*}.
\]
              Since this holds for all $I\in \mathbb{I}(F)$ and all $\xx\in F$ and $\yy\in G$,~\eqref{eq.lower.bound.facial} follows.
        \item[(b)] Once again, this follows via a similar but simpler argument to that used in part (a).  Suppose that $F,G \in \faces(\cC)$ and $F\cap G = \emptyset$.  By symmetry, to prove~\eqref{eq.lower.bound.outer.facial} it suffices to show that for $I\in \mathbb{I}(F)$
              \begin{equation}\label{eq.lower.bound.facial.equiv}
                  \dist(F,G) \ge \frac{1}{\|\sum_{i\in I\cap I_G^c} \dd_i/\sigma_{i}\|^*}.
              \end{equation}
              To that end, observe that Lemma~\ref{lemma.face} implies that for each $\uu \in \vertices(G)$ there exists $i \in I \cap I_G^c$ such that
              \[
                  \ip{\dd_i}{\uu} - e_i \ge \sigma_{i}.
              \]
              It thus follows that for all $\yy\in G$ and $\xx\in F$
              \[\sum_{i\in I\cap I_G^c} \ip{\dd_i/\sigma_{i}}{\yy - \xx} = \sum_{i\in I\cap I_G^c} (\ip{\dd_i}{\yy} - e_i)/\sigma_{i} \ge 1.\]
           Thus, by Cauchy-Schwarz inequality,
\[
                  \|\yy-\xx\| \geq \frac{1}{\|\sum_{i\in I\cap I_G^c} \dd_i/\sigma_{i}\|^*}.
\]
              Since this holds for all $I\in \mathbb{I}(F)$ and all $\xx\in F$ and $\yy\in G$,~\eqref{eq.lower.bound.facial.equiv} follows.
    \end{enumerate}
\end{proof}

\begin{corollary}\label{cor.bound.facialdist} Let $\R^n$ be endowed with a norm $\|\cdot\|$, let $\cC\subseteq\R^n$ be a polytope of the form~\eqref{eq.polyhedral},
    and let $F \in \faces(\cC)$. Then
    \[
       \bar \Phi(F,\cC) \ge \Phi(F,\cC) \ge \min_{G\in\faces(F)\atop  G\ne \cC} \max_{I\in \mathbb{I}(G)}
        \frac{1}{\|\sum_{i\in I} \dd_i/\sigma_{i}\|^*}
    \]
\end{corollary}

For polytopes with a standard form description, that is, when $D = I$ and $\ee = \mathbf{0}$, the bounds in Corollary~\ref{cor.bound.facialdist} can be simplified and sharpened as in Corollary~\ref{cor.bound.facialdist.std} below.  The statement of Corollary~\ref{cor.bound.facialdist.std} relies on the following terminology and notation.  

Let $\ssig^{-1}\in \R^k$ denote the vector with entries $1/\sigma_i, \; i\in [k]$. For $I\subseteq[k]$, let $\ssig^{-1}_I$ denote the $|I|$ subvector of $\ssig^{-1}$ obtained by selecting the entries of $\ssig^{-1}$ indexed by $I$.  We shall write $\1$ in lieu of $\ssig^{-1}$ in the special case when $\sigma_i=1$ for all $i\in [k]$.  We shall say that a norm $\|\cdot\|$ on $\R^n$ satisfies the {\em compatibility assumption} if $\|\xx\| \le \|\yy\|$ whenever $|\xx| \le |\yy|$ componentwise.  Observe that many popular norms on $\R^n$, including all $\ell_p$-norms for $1\le p\le\infty$, satisfy the compatibility assumption.

\begin{corollary}\label{cor.bound.facialdist.std} Let $\R^n$ be endowed with a norm $\|\cdot\|$ such that $\|\cdot\|^*$ 
satisfies the compatibility assumption, let $\cC\subseteq\R^n$ be a polytope of the form  $\cC = \{\xx \in \R^n \mid  A\xx = \bb, \xx \ge \mathbf{0}\}$. 
Then for all $F \in \faces(\cC)$
    \[
        \Phi(F,\cC) \ge 
        \frac{1}{\|\ssig^{-1}_{I_F}\|^*}
    \]
    and
    \[
        \bar \Phi(F,\cC) \ge \max\left\{
 \frac{1}{\|\ssig^{-1}_{I_F}\|^*}, \frac{1}{\|\ssig^{-1}_{I_{F}^c}\|^*}\right\}
    \]
        where $I_F = \{i\in[n]\colon x_i=0  \ \text{for all } \xx\in F\}$.
        
        In particular, if $\cC$ is a simplex-like polytope then for all $F \in \faces(\cC)$
        \[
        \Phi(F,\cC) \ge 
        \frac{1}{\|\1_{I_F}\|^*} \; \text{ and } \; 
        \bar \Phi(F,\cC) \ge \max\left\{
 \frac{1}{\|\1_{I_F}\|^*}, \frac{1}{\|\1_{I_F^c}\|^*}\right\}.
        \]
\end{corollary}

\begin{proof}{Proof of Proposition~\ref{prop:aff_err.face.simplex}.} This readily follows from Proposition~\ref{prop:aff_err.face}, Corollary~\ref{cor.bound.facialdist.std}, and the observation that
\[
\|\1_{I_{F(X^*)}^c}\|_2 = \sqrt{\text{\rm card}(X^*)} \text{ and } \|\1_{I_{F(X^*)}}\|_2 = \sqrt{n-\text{\rm card}(X^*)}.
\]
\end{proof}

\section{Conclusion}

We provide a template to derive affine-invariant convergence rates of the following popular versions of the Frank-Wolfe algorithm to minimize a differentiable convex function on a polytope $\cC$: vanilla, away-step, blended pairwise, and in-face.

Our template shows a convergence rate $\cO(t^{1/(2\theta-1)})$ for all of these versions of the Frank-Wolfe algorithm when a suitable {\em curvature property} and {\em error bound property of degree $\theta\in [0,1/2]$} hold. 
In the special case  $\theta = 1/2$, the expression $\cO(t^{1/(2\theta-1)})$ is to be interpreted as linear rate of convergence.  Our convergence results can also be extended to other versions of the Frank-Wolfe algorithm on polytopes that we did not discuss in this paper such as the fully-corrective and pairwise versions.

The curvature and error bound properties are both affine-invariant.  The error bound property is relative to an affine invariant distance-like function $\d: \cC\times \cC \rightarrow [0,1]$.  The suitable choice of distance function $\d$ for the affine-invariant convergence analysis of each version of the Frank-Wolfe algorithm is as follows: 
\begin{itemize}
\item {\em Radial distance} for vanilla Frank-Wolfe.
\item {\em Vertex distance} for away-step Frank-Wolfe and blended pairwise Frank-Wolfe.
\item {\em Face distance} for in-face Frank-Wolfe.
\end{itemize}

We also provide some sufficient conditions for the error bound property in terms of a H\"olderian error bound of the objective function and the following two geometric measures associated to the faces of the polytope $\cC$: the {\em inner facial distance}, and the {\em outer facial distance} from a face of $\cC$.

In addition to providing a unified derivation of convergence rates, our template recovers and extends previously known (albeit affine dependent) convergence results that had been derived separately for different versions of the Frank-Wolfe algorithm.

\subsection*{Acknowledgements}

Research reported in this paper was partially supported by the Deutsche Forschungsgemeinschaft (DFG, German Research Foundation) under Germany's Excellence Strategy – The Berlin Mathematics Research Center MATH$^+$ (EXC-2046/1, project ID 390685689, BMS Stipend).  Research reported in this paper was also partially supported by the Bajaj Family Chair at the Tepper School of Business, Carnegie Mellon University.

\bibliographystyle{plain}







\end{document}